\documentclass{amsart}

\def\bi{\begin{itemize}}
\def\ei{\end{itemize}}
\def\beq{\begin{equation}}
\def\eeq{\end{equation}}
\def\eqnok#1{(\ref{#1})}

\def\citep#1{\cite{#1}}
\def\citet#1{\cite{#1}}

\newcommand{\RK}{\mbox{\sc RK}}
\newcommand{\ARK}{\mbox{\sc ARK}}
\newcommand{\SARK}{\mbox{\sc SARK}}
\newcommand{\CG}{\mbox{\sc CG}}
\newcommand{\ARKauto}{\mbox{\sc ARK}(\mbox{\rm auto})}
\newcommand{\SARKauto}{\mbox{\sc SARK}(\mbox{\rm auto})}
\newcommand{\lambdamin}{\lambda_{\min}}
\newcommand{\lambdamax}{\lambda_{\max}}
\newcommand{\mumin}{\lambda}

\newcommand{\diag}{\mbox{\rm diag}\,}
\newcommand{\trace}{\mbox{\rm trace}\,}
\newcommand{\R}{\mathbb{R}}

\def\rcc{\color{black}}

\usepackage{subfig}
\usepackage{graphicx}
\usepackage{algorithm}
\usepackage{algorithmic}

\usepackage{color}

\newtheorem{theorem}{Theorem}[section]
\newtheorem{lemma}[theorem]{Lemma}

\theoremstyle{definition}

\theoremstyle{remark}

\numberwithin{equation}{section}

\begin{document}

\title{An Accelerated Randomized Kaczmarz Algorithm}

\author{Ji Liu}
\address{Department of Computer Sciences, University of Wisconsin-Madison, Madison, WI 53706-1685}
\curraddr{Department of Computer Sciences,
1210 W. Dayton St., Madison, WI 53706-1685}
\email{ji.liu.uwisc@gmail.com}
\thanks{The first author was supported in part by NSF Awards DMS-0914524
  and DMS-1216318 and ONR Award N00014-13-1-0129.}

\author{Stephen J. Wright}
\address{Department of Computer Sciences, University of Wisconsin-Madison, Madison, WI 53706-1685}
\email{swright@cs.wisc.edu}
\thanks{The second author was supported in part by NSF Awards DMS-0914524
  and DMS-1216318, ONR Award N00014-13-1-0129, DOE Award DE-SC0002283,
  and Subcontract 3F-30222 from Argonne National Laboratory.}


\date{\today.}



\begin{abstract}
The randomized Kaczmarz ($\RK$) algorithm is a simple but powerful
approach for solving consistent linear systems $Ax=b$. This paper
proposes an accelerated randomized Kaczmarz ($\ARK$) algorithm with
better convergence than the standard $\RK$ algorithm on ill
conditioned problems.  The per-iteration cost of $\RK$ and $\ARK$
are similar if $A$ is dense, but $\RK$ is much more able to exploit
sparsity in $A$ than is $\ARK$. To deal with the sparse case, an
efficient implementation for $\ARK$, called $\SARK$, is proposed. A
comparison of convergence rates and average per-iteration complexities
among $\RK$, $\ARK$, and $\SARK$ is given, taking into account
different levels of sparseness and conditioning. Comparisons with the
leading deterministic algorithm --- conjugate gradient applied to the
normal equations --- are also given. Finally, the analysis is
validated via computational testing.
\end{abstract}

\maketitle



\section{Introduction} \label{sec:intro}

We consider the problem of finding a solution to a consistent linear
system
\begin{equation}
Ax=b, \label{eqn_problem}
\end{equation}
where $A\in \R^{m\times n}$ and $b\in\R^{m}$. We denote the rows of
$A$ by $a_i^T$ and the elements of $b$ by $b_i$,
$i=1,2,\dotsc,m$. That is,
 \begin{equation*}
 \begin{aligned}
  A=\left[ \begin{matrix}
      a_1^T \\
      a_2^T \\
      \vdots \\
      a_m^T
   \end{matrix}\right], \quad
   b=\left[ \begin{matrix}
      b_1 \\
      b_2 \\
      \vdots \\
      b_m
   \end{matrix}\right].
      \end{aligned}
\end{equation*}
{\rcc (Our convergence results do not apply directly to inconsistent
systems.  For inconsistent systems, we can solve instead the
least-squares problem $\min_x \|Ax-b \|^2$, whose solution $x$ can be
found by solving the consistent linear system $Ax=y, A^Ty=A^Tb$.)}
Besides consistency of $Ax=b$, we assume throughout that $A$ has no
zero rows. (Such rows can be detected and eliminated in a trivial
preprocessing step.) {\rcc We assume for purposes of analysis --- though not for purposes of 
deriving and specifying the algorithms --- that the rows of $A$ are normalized:}
 \beq \label{eq:anormal} \|a_i \|_2 =
1, \;\; i=1,2,\dotsc,m.  \eeq
This assumption does not add significantly to the cost of
implementation: We could simply normalize each $a_i$ the first time it
is encountered by the algorithm. {\rcc Normalization} simplifies the analysis
in the appendix considerably, and in most cases will improve the
conditioning of the problem, leading to faster convergence. However,
in the description of algorithms in the main body of the paper, we do
not assume \eqref{eq:anormal}, and include factors $\|a_i\|^2$ as
needed. Our randomized algorithms generate the same sequence of
iterates whether or not normalization is carried out (provided, of
course, that a {\rcc corresponding} scaling is applied to $b$).

The randomized Kaczmarz ($\RK$) algorithm is an algorithm for solving
\eqnok{eqn_problem} that requires only $O(n)$ storage and has a linear
(geometric) rate of convergence. In some situations, it is even more
efficient than the conjugate gradient ($\CG$) method, which is the
most popular iterative algorithm for solving large linear systems.
At each iteration, the $\RK$ algorithm randomly selects a row $i \in
\{1,2,\dotsc,m\}$ of the linear system and does an orthogonal
projection of the current estimate vector onto the hyperplane:
\begin{equation}
x_{k+1} = x_{k} - \frac{(a_i^Tx_k-b_i)}{\|a_i\|^2}a_i.
\label{eqn:rk_step}
\end{equation}
{\rcc The $\RK$ update~\eqref{eqn:rk_step} is equivalent to one step
  of coordinate descent applied to the dual problem
\[
\min_y\quad {1\over 2}\|A^Ty\|^2 - b^Ty,
\]
(specifically, a negative gradient step in the $i$th component of $y$
with steplength $1/\|a_i\|_2^2$), where the primal variables $x$ and
duals $y$ are related through $x=A^Ty$; see \citep{Leventhal08}.}  We
denote by $i(k)$ the index selected at iteration $k$, and note that
$x_k$ depends on all the indices selected up to iteration $k$, namely,
$i(0),i(1),\dotsc,i(k-1)$.

The $\RK$ method overcomes two drawbacks of the original Kaczmarz
algorithm \citep{Kaczmarz37}. First, the original algorithm selects
rows of $A$ cyclically (not randomly) and may converge very slowly
when the data order is poor, for example, when many neighboring rows
are identical. Second, it is difficult to analyze the convergence rate
for the original Kaczmarz algorithm, whereas the expected convergence
rate of $\RK$ can be proved in a few lines.

By applying an acceleration scheme due to Nesterov to the standard
$\RK$ algorithm, we obtain an accelerated randomized Kaczmarz algorithm
($\ARK$) in Section~\ref{sec:algorithm} and show
(Section~\ref{sec:convergence}) that its linear convergence is faster
than the original method when the linear system has poor conditioning,
as measured by the minimum nonzero eigenvalue of $A^TA$. The cost per
iteration of both $\RK$ and $\ARK$ is $O(n)$ if the matrix $A$ is
dense. If $A$ is sparse, however, the calculus changes. The cost of an
iteration of $\RK$ is proportional to the number of nonzeros in $a_i$,
whereas the cost of each $\ARK$ iteration is still $O(n)$ in general.  We
therefore propose in Section~\ref{sec:equality_sparse} a scheme
called $\SARK$ in which the $\ARK$ updates are cached, to preserve
sparsity in the intermediate vectors. (In the absence of numerical
error, the iterates generated by $\ARK$ and $\SARK$ are identical.)
The average cost per iteration of $\SARK$ is $O(\sqrt{\delta}n)$,
where $\delta$ is the fraction of nonzero elements in $A$. In
Section~\ref{sec:equality_cmp_sparse}, we compare the theoretical
performance of $\RK$, $\ARK$, and $\SARK$ for different values of the
sparsity ratio $\delta$ and the minimal eigenvalue $\lambdamin$, thus
giving guidance about how to choose between these algorithms under
various scenarios. We illustrate the computational performance of the algorithm on some
random problems in Section~\ref{sec:computations}.

\subsection{Notation}

We summarize notations used in the remainder of the paper.
\begin{itemize}
\item[-] $\lambdamin$ and $\lambdamax$ are defined to be the {\rcc minimum}
  and maximum nonzero eigenvalues $A^TA$, respectively.
\item[-] $\|X\|$ is the spectral norm of the matrix, while $\|X\|_F$
  denotes the Frobenius norm.
\item[-] $X^+$ is the Moore-Penrose pseudoinverse of $X$. Denoting the
  compact singular value decomposition of $X\in \R^{m\times
    n}$ as $X=U\Sigma V^T$ where $U$ and $V$ are orthonormal matrices
  (that is, $U^TU=I$ and $V^TV=I$) and $\Sigma$ is nonsingular and
  diagonal, we have $X^+=V\Sigma^{-1}U^T$. Note that
  $\lambdamin = 1/\|(A^TA)^+\|$.
\item[-] Given a positive semidefinite matrix $M$, $\|X\|_{M}$ is
  defined as $\sqrt{\mbox{trace}(X^TMX)}$.
\item[-] Define $\mathcal{P}_{c,d}(x)$ as the {\rcc orthogonal} {\rm projection} of $x$
  onto the hyperplane given by $c^Tx=d$, that is,
\[
\mathcal{P}_{c,d}(x)=x-{c\over \|c\|^2}(c^Tx-d).
\]
\item[-] $\mathcal{P}_{A,b}(x)$ denotes the (Euclidean-norm) projection
  of $x$ onto the solution set of $Ax=b$.
\item[-] $e_j\in\R^{n}$, $j=1,2,\dotsc,n$, denotes the $j$th Euclidean
  basis vector --- a vector of $n$ zeros except for $1$ in position
  $j$.
\end{itemize}

\section{Related Work}

The Kaczmarz algorithm was proposed by Kaczmarz \citet{Kaczmarz37},
who used the cyclic projection procedure to solve consistent linear
systems $Ax=b$. He proved the convergence to the unique solution if
$A$ is a square nonsingular matrix. The cyclic ordering of the
iterates made it difficult to obtain iteration-based convergence
results, but Galantai~\citet{Galantai25} proved a linear convergence
rate in terms of cycles. {\rcc Since the 1980s, the Kaczmarz algorithm
  has found an important application area in Algebraic Reconstruction
  Techniques (ART) for image reconstruction; see for example
  \citet{Herman80} and \citet{Herman09}.}

Censor et al. \citet{Cegogo01} proposed a component averaging method
to solve~\eqref{eqn_problem}: Parallel-project the current $x$ onto
all hyperplanes and apply an average scheme on all projections to
obtain the next iterate $x$. This method is essentially a gradient
descent method for solving ${1\over 2}\|Ax-b\|^2$, and can thus handle
inconsistent systems.

Strohmer and Vershynin \citet{Strohmer09} studied the behavior of
$\RK$ in the case of a consistent system $Ax=b$ in which $A$ has full
column rank (making the solution unique).  They proved the linear
convergence rate for $\RK$ in expectation. Needell \cite{Needell10}
also assumed full column rank, but dropped the assumption of
consistency, showing that the $\RK$ algorithm converges linearly to a
ball of fixed radius centered at the solution. The radius is
proportional to the distance of $b$ from the image space of $A$.
Eldar and Needell \citet{EldarN11} presented a modified version of the
randomized Kaczmarz method which at each iteration selects the optimal
projection from a randomly chosen set. This technique improves the
convergence rate but requires more computation cost in each iteration.


Leventhal and Lewis \citet{Leventhal08} extended the $\RK$ algorithm
for consistent linear equalities $Ax=b$ to the more general setting of
consistent linear inequalities and equalities: $A_Ix\geq b_I$,
$A_Ex=b_E$. The basic idea is quite similar to the $\RK$ algorithm:
iteratively update $x_{k+1}$ by projecting $x_{k}$ onto the randomly
selected hyperplane or half space. The linear convergence rate was
proven to be $1-1/{(L^2\|A\|^2_F)}$, where $L$ is the Hoffman constant
\citep{Hoffman52} for the system $A_Ix\geq b_I$, $A_Ex=b_E$.


Zouzias and Freris \citet{Zouzias12} considered the case of possibly
inconsistent \eqnok{eqn_problem}. They proposed a randomized extended
Kaczmarz algorithm by first projecting $b$ orthogonally onto the image
space of $A$ to obtain $b_\bot$, then orthogonally projecting the
initial point $x_0$ onto the hyperplane $Ax=b_{\bot}$. Essentially,
the $\RK$ algorithm is applied twice. The convergence rate is proven
to be $1-{\lambdamin/\|A\|^2_F}$, which is the same as the
$\RK$ algorithm for consistent linear systems. This method can be
considered as a randomized variant of the extended Kaczmarz method
proposed by Popa~\cite{Popa99}.

\section{Algorithm} \label{sec:algorithm}

In this section, we review the randomized Kaczmarz algorithm ($\RK$,
Algorithm~\ref{alg_rka}) and propose an accelerated variant called
$\ARK$ (Algorithm~\ref{alg_arka}). Finally, we describe an equivalent
version of $\ARK$ that can be implemented with fewer operations
(Algorithm~\ref{alg_arka2}).

Each iteration of $\RK$ randomly selects a hyperplane $a_i^Tx=b_i$,
for some $i \in \{1,2,\dotsc,m\}$, and obtains $x_{k+1}$ by {\rcc
  orthogonally} projecting $x_k$ onto this hyperplane. As shown at the
start of Section~\ref{sec:convergence}, this algorithm guarantees
linear convergence in the expectation sense.


\begin{algorithm}                     
\caption{Randomized Kaczmarz: $x_{K+1}=\RK(A,b,x_0,K)$}          
\label{alg_rka}                           
\begin{algorithmic}[1]                    
\STATE Initialize $k \leftarrow 0$;
\WHILE{$k \leq K$}
\STATE Choose $i=i(k)$ from $\{1,2,3,\dotsc,m\}$ with equal probability;
\label{eqn_alg_rka_2}
\STATE Set $x_{k+1} \leftarrow \mathcal{P}_{a_i,b_i}(x_k)$, that is,
$x_{k+1}=x_k- a_i(a_i^Tx_k-b_i) / \|a_i\|^2$;
\label{eqn_alg_rka_3}
\STATE $k \leftarrow k+1$;
\ENDWHILE
\end{algorithmic}
\end{algorithm}

{\rcc We note again that Step~\ref{eqn_alg_rka_3} does not change if
  we omit the normalization step, that is, if $\|a_i \| \neq
  1$. However, when the rows are not normalized, our RK algorithm
  becomes inconsistent with the versions described in
  \cite{Leventhal08,Strohmer09}, which select the index $i$ in
  Step~\ref{eqn_alg_rka_2} with probability $\| a_i \|_2^2 / \|A
  \|_F^2$.  We could simulate the effects of non-normalized rows by
  defining a matrix $\bar{A}$ in which row $a_i$ is replaced by $\|a_i
  \|^2$ copies of the normalized rows $a_i/\|a_i\|$. Our
  Algorithm~\ref{alg_rka} applied to this virtual matrix $\bar{A}$
  would then be equivalent to Algorithm~1 of \cite{Strohmer09} applied
  to $A$, with the same convergence results as in that paper
(see \eqref{eqn_rk_bound},
  with $\|A\|_F^2$ replacing $m$ in the denominator of the rate
  constant).  Analysis of the accelerated algorithm to be discussed
  below could also be performed without the assumption of
  normalization, but the situation becomes considerably more complicated in this
  case. In particular, several subtle issues related to allowable
  scalings of $A$ (not dealt with in existing
  analyses of accelerated methods) must be addressed. We believe that
  any additional generality to be gained by dropping our assumptions
  of normalization and uniform probabilities is minor, and would be
  obscured by the additional complication in the analysis.  }

The $\ARK$ algorithm applies Nesterov's accelerated procedure
\citep{nesterov2004introductory} --- more familiar in the context of
gradient descent for optimization --- to the standard $\RK$
algorithm. When applied to $\min_x \, f(x)$, gradient descent sets
$x_{k+1} \leftarrow x_k-\theta_k\nabla f(x_k)$, where $\nabla f$ is
the objective gradient and $\theta_k$ is the stepsize. Nesterov's
accelerated procedure introduces two sequences $\{y_k\}$ and $\{v_k\}$
and defines the following iterative scheme:
\begin{align*}
  y_k & \leftarrow  \alpha_k v_k+(1-\alpha_k)x_k\\
  x_{k+1} & \leftarrow   y_k - \theta_k\nabla f(y_k)\\
  v_{k+1} & \leftarrow \beta_k v_k + (1-\beta_k)y_k -\gamma_k\nabla f(y_k).
\end{align*}
With appropriate {\rcc choices} of $\alpha_k$, $\beta_k$, and $\gamma_k$, this
procedure yields better convergence rates than standard gradient
descent.

If we treat the projection operation of Step~\ref{eqn_alg_rka_3} in
Algorithm~\ref{alg_rka} analogously to the gradient descent step,
we can obtain an accelerated version of the $\RK$ algorithm.
This accelerated randomized Kaczmarz ($\ARK$) procedure is detailed in
Algorithm~\ref{alg_arka}.
The scalars $\alpha_k$, $\beta_k$, and $\gamma_k$ in
Algorithm~\ref{alg_arka} are independent of the vector sequences
$\{x_k\}$, $\{y_k\}$, and $\{v_k\}$, and can be calculated offline.


\begin{algorithm}
\caption{Accelerated Randomized Kaczmarz: $x_{K+1}=\ARK(A,b,\mumin, x_0,K)$}          
\label{alg_arka}                           
\begin{algorithmic}[1]                    
\STATE {\rcc Check that $\lambda \in [0,\lambdamin]$;}
\STATE Initialize $v_0 \leftarrow x_0$, $\gamma_{-1} \leftarrow 0$, $k \leftarrow 0$;
\WHILE{$k \leq K$}
\STATE Choose $\gamma_k$ to be  the larger root of
\begin{equation} \gamma_k^2-{\gamma_k \over m}=\left(1-{\gamma_k\mumin\over m} \right) \gamma_{k-1}^2; \label{eqn_alg1}
\end{equation}
\STATE Set $\alpha_k$ and $\beta_k$ as follows:
\begin{equation} \alpha_k \leftarrow {m-\gamma_k\mumin\over \gamma_k(m^2-\mumin)},\label{eqn_alg2}
\end{equation}
\begin{equation} \beta_k \leftarrow 1-{\gamma_k\mumin\over m}; \label{eqn_alg3}\end{equation}
\STATE Set $y_k \leftarrow \alpha_kv_k+(1-\alpha_k)x_k$; \label{alg_arka_y}
\STATE Choose $i=i(k)$ from $\{1,2,3,\dotsc,m\}$ with equal probability;
\STATE Set $x_{k+1} \leftarrow \mathcal{P}_{a_i,b_i}(y_k)$, that is,
$x_{k+1} \leftarrow y_k-a_i(a_i^Ty_k-b_i) / \|a_i\|^2$;
\label{alg_arka_x}
\STATE Set
$v_{k+1} \leftarrow \beta_kv_k+(1-\beta_k)y_k-\gamma_k a_i(a_i^Ty_k-b_i)/
\|a_i\|^2$;\label{alg_arka_v}
\STATE $k \leftarrow k+1$;
\ENDWHILE
\end{algorithmic}
\end{algorithm}


We now describe the complexity of these methods for the case of dense
$A$. We make the standing assumption that the quantities $\|a_i \|^2$,
$i=1,2,\dotsc,m$ are precomputed via an initial pass through the
matrix.  The main computation in Algorithm~\ref{alg_rka} is in step
\ref{eqn_alg_rka_3}, which requires about $4n$ operations per
iteration.  The cost per iteration of Algorithm~\ref{alg_arka} is
about $11n$, incurred in steps \ref{alg_arka_y}, \ref{alg_arka_x}, and
\ref{alg_arka_v}.

Although Algorithm~\ref{alg_arka} is useful for purposes of
convergence analysis of the accelerated Kaczmarz algorithm, we
describe an equivalent implementation in Algorithm~\ref{alg_arka2}
that has a lower cost per iteration.
Denoting
\[
g_k := a_i(a^T_iy_k-b_i)/\|a_i\|^2,
\]
we have $x_{k+1} = y_k-g_k$.  Substituting for $v_k$ from
Step~\ref{alg_arka_y} into Step~\ref{alg_arka_v}, we obtain
\begin{align*}
v_{k+1} & = \beta_{k}v_k+(1-\beta_k)y_k - \gamma_kg_k\\
& = \left({\beta_{k}\over \alpha_k} + 1 -\beta_k\right) y_k -
\beta_k{1-\alpha_k\over \alpha_k}x_k-\gamma_kg_k.
\end{align*}
By substituting for $v_{k+1}$ in Step~\ref{alg_arka_y} of
Algorithm~\ref{alg_arka}, for iterate $k+1$, we obtain
\begin{align}
\nonumber
y_{k+1} &= \alpha_{k+1}v_{k+1} + (1-\alpha_{k+1})x_{k+1}\\
\nonumber
&= \alpha_{k+1}\left({\beta_k \over \alpha_k}+1-\beta_k \right)y_k - \alpha_{k+1}\beta_k{1-\alpha_k\over \alpha_k}x_k \\
\nonumber
& \quad\quad
 - \alpha_{k+1}\gamma_k g_k + (1-\alpha_{k+1})(y_k-g_k)\\
\nonumber
&=\left[1{\rcc +} \alpha_{k+1}\beta_k\left({1-\alpha_k\over \alpha_k}
\right)\right]y_k - \alpha_{k+1}\beta_k{1-\alpha_k \over
\alpha_k}x_k \\
\label{eqn_y}
& \quad\quad - \left(1-\alpha_{k+1}+\alpha_{k+1} \gamma_k\right)g_k.
\end{align}
From \eqref{eqn_alg2} and \eqref{eqn_alg3}, we have
\begin{align*}
 -\beta_k\left({1-\alpha_k\over \alpha_k}\right) &= {m-\gamma_k\mumin \over m}\left(1-{\gamma_k(m^2-\mumin)\over m-\gamma_k\mumin}\right)\\
&={m-\gamma_k\mumin\over m} - {\gamma_k(m^2-\mumin)\over m}\\
&=1-m\gamma_k.
\end{align*}
Thus, by substituting into \eqref{eqn_y}, we obtain
\[
y_{k+1}=(1-m\gamma_k)\alpha_{k+1}x_k + (1-\alpha_{k+1}+m\alpha_{k+1}\gamma_k)y_k-(1-\alpha_{k+1}+\alpha_{k+1}\gamma_k)g_k.
\]

By making these substitutions into Algorithm~\ref{alg_arka}, we obtain
the equivalent implementation of  Algorithm~\ref{alg_arka2}.

\begin{algorithm} [htp]                     
\caption{Efficient $\ARK$: $x_{K+1}=\ARK(A,b,\mumin, x_0,K)$}          
\label{alg_arka2}                           
\begin{algorithmic}[1]                    
\STATE {\rcc Check that $\lambda \in [0,\lambdamin]$;}
\STATE Initialize $y_0=x_0$, $\gamma_{-1}=0$, $k=0$. 
\STATE Generate the sequences $\{\gamma_k:~k=0,1,\dotsc,K+1\}$ and
$\{\alpha_{k}:~k=0,1,\dotsc,K+1\}$
as in \eqnok{eqn_alg1} and \eqnok{eqn_alg2};
\WHILE{$k \leq K$}
\STATE Choose $i=i(k)$ from $\{1,2,3,\dotsc,m\}$ with equal probability;
\STATE Set
$s_k \leftarrow (a_i^Ty_k-b_i)/\|a_i\|^2$; \label{alg_arka2_0}
\STATE Set $g_k  \leftarrow
s_ka_i$; \label{alg_arka2_1}
\STATE Set $g'_k \leftarrow
\left(1-\alpha_{k+1}+\alpha_{k+1} \gamma_k\right)s_ka_i$;
\label{alg_arka2_2}
\STATE Set
$y'_{k+1} \leftarrow (1-m\gamma_k)\alpha_{k+1}x_k+(1-\alpha_{k+1}+m\alpha_{k+1}\gamma_k)y_k$;
\label{alg_arka2_3}
\STATE Set $y_{k+1} \leftarrow y'_{k+1} -g'_k$;
\label{alg_arka2_4}
\STATE Set $x_{k+1} \leftarrow y_k-g_k$;
\label{alg_arka2_5}
\STATE $k \leftarrow k+1$;
\ENDWHILE
\end{algorithmic}
\end{algorithm}

The main computations are in Step~\ref{alg_arka2_0} to
Step~\ref{alg_arka2_5} which have operation counts of about $2n$, $n$,
$n$, $3n$, $n$, and $n$, respectively, giving a total of $9n$. If
parallel computation is possible, then Steps~\ref{alg_arka2_0},
\ref{alg_arka2_1}, and \ref{alg_arka2_2} can be performed
simultaneously with Step~\ref{alg_arka2_3} (in time complexity about
$3n$) while Steps~\ref{alg_arka2_4} and \ref{alg_arka2_5} can be
performed simultaneously (in time about $n$).  In this setting, the
total complexity can be reduced to about $4n$ --- a count identical to
the $\RK$ algorithm.

\section{Efficient Implementation for Sparse Data}\label{sec:equality_sparse}

This section considers the case in which the data matrix $A$ is
sparse, with a fraction of $\delta$ nonzeros (with $0 < \delta \ll 1$)
and seeks an efficient implementation of Algorithm~\ref{alg_arka2} for
this case.  We assume that the nonzeros are not concentrated in
certain rows of $A$, that is, the sparsity of each row $a_i^T$ is also
approximately $\delta$.

Note that the $\ARK$ approach starts at a significant disadvantage in
the sparse setting. While sparsity can be exploited easily in $\RK$
--- the average number of operations for each iteration of
Algorithm~\ref{alg_rka} is approximately {\rcc $4 \delta n$} --- the
operation counts of the $\ARK$ algorithms remain at $O(n)$, since the
vectors $x_k$, $y_k$, and $v_k$ are dense in
general. {\rcc (Algorithm~\ref{alg_arka2} has a count of approximately $3n+6
\delta n$ per iteration.)} We now seek a modification of
Algorithm~\ref{alg_arka2} that ``caches'' the updates in order to
maintain some sparsity in the update vectors, thus reducing the
average complexity of each $\ARK$ iteration.

We start by writing the main updating steps in
Algorithm~\ref{alg_arka2} as follows:
\begin{subequations} \label{eqn_fastxy}
\begin{align}
s_k &= (a_{i(k)}^Ty_k-b_{i(k)})/\|a_{i(k)}\|^2, \\
x_{k+1} &= y_k - s_ka_{i(k)}, \\
y_{k+1} &= P_kx_k + Q_ky_k - R_ks_k {\rcc a_{i(k)}},
\end{align}
\end{subequations}
where
\begin{subequations} \label{eqn_PQR}
\begin{align}
P_k &= \alpha_{k+1}(1-m\gamma_k), \\
Q_k &= 1-\alpha_{k+1}+m\alpha_{k+1}\gamma_k, \\
R_k &= 1-\alpha_{k+1}+\alpha_{k+1}\gamma_k.
\end{align}
\end{subequations}
Since updating $x_{k+1}$ and $y_{k+1}$ is quite expensive, we only
update them once on each cycle (that is, once per $T$ iterations).
{\rcc We see by recursive application of \eqnok{eqn_fastxy} that each
  iterate $x_{k+t}$, $y_{k+t}$ for $t \ge 1$ can be expressed as a
  linear combination of $x_k$ and $y_k$, plus one other vector. The
successive updates from rows $a_{i(k)}, a_{i(k+1)}, \dotsc,
a_{i(k+t)}$ can be ``cached'' in vectors $z_t$ and $w_t$, so that
$x_{k+t}$ and $y_{k+t}$ can be written as follows:}
\begin{subequations} \label{eq:xykt}
\begin{align}
x_{k+t}=& \rho_tx_k + \tau_t y_k + z_t, \\
y_{k+t}=&\sigma_tx_k + \nu_t y_k + w_t,
\end{align}
\end{subequations}
where $\rho_t$, $\tau_t$, $\sigma_t$, and $\nu_t$ are {\rcc scalars. Rather
than forming $x_{k+t}$ and $y_{k+t}$ explicitly, we could instead
update the quantities $\rho_t$, $\tau_t$, $\sigma_t$, $\nu_t$, $z_t$,
and $w_t$ at each iteration.} The advantage of doing so is that,
provided $t$ is not too large, the vectors $z_t$ and $w_t$ are not
dense, so the cost of updating this implicit representation is {\rcc usually}
lower than the explicit version. At some point, when $t$ grows too
large, the vectors $z_t$ and $w_t$ ``fill in'' {\rcc enough} that the
advantages of implicit representation are lost. At this point ---
after $T$ steps, say --- we can store the latest vectors $x_{k+T}$ and
$y_{k+T}$ explicitly, and start a new cycle of $T$ iterations.


We now obtain the update formulae for the quantities $\rho_t$,
$\tau_t$, $\sigma_t$, $\nu_t$, $z_t$, and $w_t$. At the starting point
of a cycle, we set $t=0$ and
\[
 \rho_0 = 1, \tau_0 = 0, \sigma_0 = 0, \nu_0=1, z_0=w_0=0,
\]
{\rcc so that \eqref{eq:xykt} holds for $t=0$.}  In the step from iteration
$t$ to iteration $(t+1)$ of a cycle, we have
\[
 x_{k+t+1} = y_{k+t} - {\rcc s_{k+t} a_{i(k+t)}}=\sigma_t x_k + \nu_t y_k + w_t - s_{k+t} {\rcc a_{i (k+t)}},
\]
implying that
\begin{align*}
 \rho_{t+1} &= \sigma_t, \\
\tau_{t+1} &= \nu_t, \\
 z_{t+1} &= w_t-s_{k+t}{\rcc a_{i(k+t)}}.
\end{align*}
Similarly, from
\begin{align*}
 y_{k+t+1} &=P_{k+t}x_{k+t} + Q_{k+t}y_{k+t} - R_{k+t}s_{k+t}{\rcc a_{i(k+t)}}\\
& =P_{k+t}(\rho_tx_k+\tau_ty_k+z_t) + Q_{k+t}(\sigma_tx_k+\nu_ty_k+w_t) - R_{k+t}s_{k+t}{\rcc a_{i(k+t)}}\\
&=(P_{k+t}\rho_t+Q_{k+t}\sigma_t)x_{k} + (P_{k+t}\tau_t +
Q_{k+t}\nu_t)y_{k} \\
& \quad\quad + (P_{k+t}z_t +
Q_{k+t}w_t-R_{k+t}s_{k+t}{\rcc a_{i(k+t)}}),
\end{align*}
we have
\begin{align*}
 \sigma_{t+1}& =P_{k+t}\rho_t+Q_{k+t}\sigma_t, \\
\nu_{t+1}&=P_{k+t}\tau_t + Q_{k+t}\nu_t, \\
w_{t+1}&=P_{k+t}z_t + Q_{k+t}w_t-R_{k+t}s_{k+t}{\rcc a_{i(k+t)}}.
\end{align*}
The scalar $s_{k+t}$ can be computed from
\begin{align}
\notag
 s_{k+t} & = {\rcc (a^T_{i(k+t)}}y_{k+t}-{\rcc b_{i(k+t)}) / \| a_{i(k+t)} \|^2} \\
\notag
& = {\rcc [a_{i(k+t)}^T(\sigma_tx_k+\nu_ty_k+w_t)- b_{i(k+t)}] / \| a_{i(k+t)} \|^2} \\
\label{eq:skt}
& = {\rcc [\sigma_t a_{i(k+t)}^Tx_k+
\nu_t a_{i(k+t)}^T y_k
+ a_{i(k+t)}^T w_t - b_{i(k+t)}]/\| a_{i(k+t)} \|^2}.
\end{align}
We show this approach in full detail, for cycles of fixed length $T$,
in Algorithm~\ref{alg_sark}.

\begin{algorithm} [t!]                     
\caption{Efficient  $\ARK$ for Sparse $A$: $x_{K+1}=\SARK(A, b, \mumin, x_0, T, K)$ }          
\label{alg_sark}                           
\begin{algorithmic}[1]                    
\STATE{\rcc  Check that $\lambda \in [0,\lambdamin]$;}
\STATE Initialize $y_0=x_0$, $\gamma_{-1}=0$, $k=0$;
\STATE Generate the sequences
\begin{alignat*}{2}
P_k &=\alpha_{k+1}(1-m\gamma_k), \;\; &&k=0,1,\dotsc,K, \\
Q_k &=1-\alpha_{k+1}+m\alpha_{k+1}\gamma_k, \;\; && k=0,1,\dotsc,K, \\
R_k &=1-\alpha_{k+1}+\alpha_{k+1}\gamma_k, \;\; && k=0,1,\dotsc,K;
\end{alignat*}
\WHILE{$k \leq K$}
\STATE Set $t\leftarrow 0$, $\bar{x}\leftarrow x_k$, $\bar{y}\leftarrow y_k$, $\rho_0\leftarrow 1$,
$\tau_0\leftarrow 0$, $\sigma_0\leftarrow 0$, $\nu_0\leftarrow 1$, $z_k\leftarrow 0$, $w_k\leftarrow 0$;
\WHILE{$t<T$}
\IF{$k+t\geq K$}
\STATE break;
\ENDIF
\STATE Choose $i=i(k)$ from
$\{1,2,3,\dotsc,m\}$ with equal probability;
\STATE Set
\begin{align*}
s_{k+t}&=(a_{i}^T(\sigma_t\bar{x}+\nu_t\bar{y}+w_t)-b_{i})/\|a_i\|^2,\\
 \rho_{t+1}&=\sigma_t,\\
 \tau_{t+1}&=\nu_t,\\
 \sigma_{t+1}&=P_{k+t}\rho_t+Q_{k+t}\sigma_t,\\
\nu_{t+1}&=P_{k+t}\tau_t + Q_{k+t}\nu_t,\\
 z_{t+1}&=w_t-s_{k+t}a_{i},\\
w_{t+1}&=P_{k+t}z_t + Q_{k+t}w_t-R_{k+t}s_{k+t}a_{i};
\end{align*}
\STATE Set $t \leftarrow t+1$;
 \ENDWHILE
\STATE Set $k \leftarrow k+T$;
 \STATE \label{alg_sark_endcycle} Set
\begin{align*}
 x_k &=\rho_t \bar{x} +\tau_t \bar{y} + z_t, \\
  y_k &=\sigma_t\bar{x} + \nu_t \bar{y} + w_t;
\end{align*}
\ENDWHILE
\end{algorithmic}
\end{algorithm}

Note that $w_{t}$ and $z_{t}$ have nonzeros in locations where any of
the vectors \\ $ {\rcc a_{i(k)}, a_{i(k+1)}, \dotsc, a_{i(k+t-1)}} $
contain nonzeros. Thus, assuming that these vectors do not overlap
significantly, and that each of them has about $\delta n$ nonzeros, we
can estimate that $w_t$ and $z_t$ have about $t \delta n$ nonzeros, {\rcc in
the same locations as each other. The major costs at each iteration
are as follows:
\begin{itemize}
\item[-] $s_{k+t}$ costs about $6\delta n$ operations when evaluated
  according to \eqref{eq:skt}, since $a_{i(k+t)}$ has about $\delta n$
  nonzeros.
\item[-] $z_{t+1}$ costs about $2\delta n$ operations, for the same
  reason.
\item[-] $w_{t+1}$ costs about $3t\delta n+2\delta n$ operations, since
  $z_t$ and $w_t$ both have about $\delta nt$ nonzeros, in the same
  locations, and $a_{i(k+t)}$ has about $\delta n$ nonzeros.
\end{itemize}
} 
The cost of updating $x_k$ and $y_k$ in Step~\ref{alg_sark_endcycle}
is about $3n+T\delta n$ each. Therefore, over a complete cycle of $T$
iterations, we expect an approximate operation count of
\[
\sum_{t=1}^{T-1}\left(3t\delta n+10\delta n\right) + 6n +
2T\delta n \approx 1.5(T-1)T\delta n+6n+12T\delta n,
\]
giving an approximate average cost per iteration of
\[
1.5(T-1)\delta n + {6n\over T} + 12\delta n.
\]
This count is minimized by setting $T=T^*=2/\sqrt{\delta}$; for this
value we obtain an average count per iteration of
$6\sqrt{\delta}n+10.5\delta n$.
This is still worse than the iteration cost for $\RK$ (which is
$O(\delta n)$) but much better than that of $\ARK$ (which is
$O(n)$). We show in the next section that the total number of
iterations required by $\ARK$ to achieve a prescribed accuracy is
lower {\rcc than for $\RK$}, in general, which makes Algorithm~\ref{alg_sark}
competitive in some regimes.

\section{Convergence Rate} \label{sec:convergence}


In this section, we study the convergence behavior of
Algorithms~\ref{alg_rka}, \ref{alg_arka2}, and \ref{alg_sark},
estimating in particular the total number of operations required to
achieve a specified level of accuracy.  We also compare the approach
with the conjugate gradient ($\CG$) algorithm, applied to the ``normal
equations'' system $A^TAx=A^Tb$.
To simplify the comparisons, we assume throughout that
\eqnok{eq:anormal} holds, so that  $\|A\|_F^2 = m$.

The convergence of $\RK$ (Algorithm~\ref{alg_rka}) is studied in
\citep{Strohmer09, Leventhal08}\footnote{In \cite{Strohmer09}, it is
  required that $A$ has full column rank, but this requirement is
  removed in \cite [Theorem 4.3]{Leventhal08}, where the Hoffman
  constant $L$ is equivalent to ${1/\sqrt{\lambdamin}}$.} It is shown that
\begin{equation} \label{eqn_rk_bound}
\mathbb{E}(\|x_{k+1}-\mathcal{P}_{A,b}(x_{k+1})\|^2) \leq
\left(1-{\lambdamin \over m}\right)^{k+1}\|x_0-\mathcal{P}_{A,b}(x_0)\|^2,
\end{equation}
where the expectation is taken over the indices
$i(0),i(1),i(2),\dotsc$ selected at each iteration.

For $\ARK$, we have the following result. The proof can be found in
the appendix. It is quite technical, and follows to some extent the
framework developed by Nesterov \cite{Nesterov10} for the accelerated
coordinate descent method.
\begin{theorem} \label{thm_main}
Apply $\ARK$ to the problem \eqref{eqn_problem} with $\mumin
\in [0,{\lambdamin}]$, and define $\sigma_1=1+{\sqrt{\mumin}\over 2m}$
and $\sigma_2=1-{\sqrt{\mumin}\over 2m}$.  Then we have for any $k\geq
0$ that
\begin{equation}
\mathbb{E}(\|v_{k+1}-x^*\|^2_{(A^TA)^+}) \leq {4\|x_0-x^*\|_{(A^TA)^+}^2 \over (\sigma_1^{k+1} +\sigma_2^{k+1})^{2}} \label{eqn_thm1}
\end{equation}
and
\begin{equation}
\mathbb{E}(\|x_{k+1}-x^*\|^2) \leq {4\mumin \|x_0-x^*\|_{(A^TA)^+}^2 \over (\sigma_1^{k+1} - \sigma_2^{k+1})^{2}}, \label{eqn_thm2}
\end{equation}
where
$
x^*:=\mathcal{P}_{A,b}(x_0)=x_0+A^+(b-Ax_0).
$
\end{theorem}

Essentially, Theorem~\ref{thm_main} ensures that the $\ARK$ algorithm
converges in expectation to the projection of the initial point $x_0$
onto the affine space defined by $Ax=b$.

Theorem~\ref{thm_main} shows that when $\mumin>0$, the $\ARK$
algorithm converges at a linear rate. If the value of $\mumin=0$, we
can obtain a sublinear rate. {\rcc By taking limits as
$\lambda \to 0^+$ in \eqref{eqn_thm2}, we have}
\begin{align}
\nonumber
&\lim_{\mumin\rightarrow 0^+} {4\mumin \|x_0-x^*\|_{(A^TA)^+}^2 \over (\sigma_1^{k+1} - \sigma_2^{k+1})^{2}}\\
\nonumber
& =\lim_{\mumin\rightarrow 0^+} {4\mumin \|x_0-x^*\|_{(A^TA)^+}^2 \over \left(\left(1+{(k+1)\sqrt{\mumin}\over 2m} + o(\sqrt{\mumin})\right) - \left(1-{(k+1)\sqrt{\mumin}\over 2m} + o(\sqrt{\mumin})\right)\right)^{2}}\\
\nonumber
&=\lim_{\mumin\rightarrow 0^+} \frac{4\mumin \|x_0-x^*\|_{(A^TA)^+}^2}{\left({{(k+1)\sqrt{\mumin}\over m} + o(\sqrt{\mumin})}\right)^2}\\
\label{eq:ark_sublin}
&=\frac{4m^2\|x_0-x^*\|_{(A^TA)^+}^2}{(k+1)^2}.
\end{align}

Next, we compare convergence rates of $\RK$, $\ARK$, and
$\CG$.  We assume further that $\mumin$ is set to its
optimal value $\lambdamin$ in $\ARK$. Since all algorithms
converge rapidly when $\lambdamin({A^TA})$ is large, we are
particularly interested in the case in which $\lambdamin$ is
small, that is, the linear system is ill-conditioned.

\subsection{Comparison between $\RK$ and $\ARK$ for Dense $A$}
\label{sec:equality_cmp}

The right-hand side of the bound \eqref{eqn_rk_bound} decreases by a
factor of $1-\mumin/ m$ at each iteration.  For $\ARK$, we
have that $\sigma_2^k \to 0$, so the decrease of the right-hand side
is governed mainly by the behavior of the $\sigma_1$ term in the
denominator. Asymptotically, we have a decrease factor per iteration
of approximately
\beq \label{eqn_sigma1}
\sigma_1^{-2} =   \left( 1+ \frac{\sqrt{\mumin}}{2m}\right)^{-2} \approx
1-\frac{\sqrt{\mumin}}{m}.
\eeq
We conclude that for small values of $\mumin$, the $\ARK$ approach
will have significantly faster linear convergence. Even if we measure
convergence rate {\em per operation}, $\ARK$ is still faster in general,
since in the implementation of Algorithm~\ref{alg_arka2}, it requires
only twice as many operations per iteration as $\RK$.

\subsection{Comparison among $\RK$, $\ARK$, and $\SARK$ for Sparse $A$}
\label{sec:equality_cmp_sparse}

When the coefficient matrix $A$ is sparse, the comparisons change,
because each iteration of $\RK$ costs less than each iteration of
either $\ARK$ or $\SARK$. On the other hand, fewer iterations of
$\ARK$ are required to reduce the expected error below a specified
tolerance. From \eqnok{eqn_rk_bound}, we deduce that the number $N$ of
iterations needed to reduce
$\mathbb{E}(\|x_N-\mathcal{P}_{A,b}(x_N)\|^2)$ below a target
threshold $\epsilon$ is $O((m/\lambdamin) | \log \epsilon |)$.
We have from \eqnok{eqn_thm2} and \eqnok{eqn_sigma1} (and ignoring a
$\log \mumin$ term) that the number of iterations $N$ of
$\ARK$ and $\SARK$ needed to reduce $\mathbb{E}(\|x_k-x^*\|^2)$ below
$\epsilon$ is $O((m/\sqrt{\mumin}) | \log \epsilon |)$. Assuming
approximately $\delta n$ nonzeros in each row of each row of $A$, we
summarize the operation and iteration counts for $\RK$, $\ARK$, and
$\SARK$ in Table~\ref{tab_e}.
\begin{table}
\centering
  \begin{tabular}{lcc}
    \hline
& Approx Operations per Iteration & Approx Iterations \\ \hline\hline
    $\RK$ {\rcc (Algorithm~\ref{alg_rka})} & $4\delta n$ & $|\log\epsilon|(m /\lambdamin)$ \\ \hline
    $\ARK$ {\rcc (Algorithm~\ref{alg_arka2})} & $3n+6\delta n $ & $|\log\epsilon| (m / \sqrt{\mumin})$ \\ \hline
    $\SARK$ {\rcc (Algorithm~\ref{alg_sark})} & $6\sqrt{\delta}n +10.5\delta n $ & $|\log\epsilon| (m / \sqrt{\mumin})$ \\
    \hline
  \end{tabular}
\caption{Operation and Iteration Counts for to achieve expected
  accuracy $\epsilon$ for randomized Kaczmarz variants.} \label{tab_e}
\end{table}

From the data in Table~\ref{tab_e}, and assuming that $\mumin$ is
set to its optimal value $\lambdamin$ in the $\ARK$ and $\SARK$
algorithms, we conclude the following about the relative performance
of these three approaches for various values of $\delta$ and
$\lambdamin$.
\begin{itemize}
 \item[-] $\RK$ will be approximately the best option if
\[
\mumin\geq \max \left\{\left({4\delta\over
  3+6\delta}\right)^2,~\left({4\sqrt{\delta}\over
  6+10.5\sqrt{\delta}}\right)^2\right\};
\]
\item[-] $\SARK$ will be approximately best if
\[
\mumin\leq \left({4\sqrt{\delta}\over
  6+10.5\sqrt{\delta}}\right)^2~\text{and}~\delta \leq 0.1;
\]
\item[-] $\ARK$ will be approximately best, otherwise.
\end{itemize}
We illustrate these claims in Figure~\ref{fig_sar}.  Note that our
comparison is based on approximate and worst-case analyses, which is
why we claim only ``approximate'' superiority for each set of
values in question.
{\rcc We can confidently say, however, $\RK$ will be superior for
  larger values of $\lambdamin$, while $\ARK$ favors small
  $\lambdamin$ and large $\delta$, and $\SARK$ is superior to $\ARK$
  for small values of $\delta$. For small fixed values of
  $\lambdamin$, $\RK$ will be superior for small values of $\delta$,
  then $\SARK$ will be superior for intermediate $\delta$ values, and
  $\ARK$ superior for larger $\delta$ values.  }

\begin{figure}[htp!]
  \centering
    \subfloat{\includegraphics[width=0.8\textwidth]{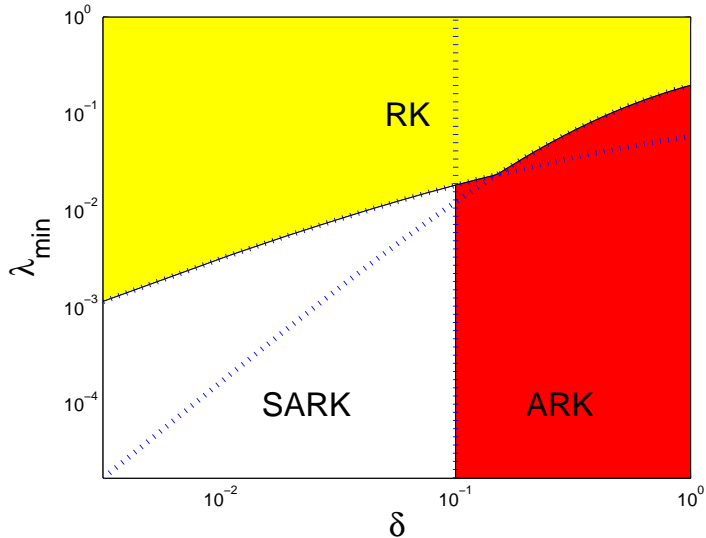}}
    \caption{Illustration of the regions of the
      $(\delta,\lambdamin)$ space for which $\RK$, $\ARK$, and
      $\SARK$ are approximately superior (that is, $\lambda_{\min}$ in
      the graph). The white (yellow, red) area indicates that $\SARK$
      ($\RK$, $\ARK$) is approximately best for the given combination
      of values.\label{fig_sar}}
\end{figure}


\subsection{Comparison among $\RK$, $\SARK$, and $\CG$}
\label{sec:boundcmp_cg}

We next compare $\RK$ and $\ARK$ with conjugate gradient ($\CG$)
applied to the normal-equations system $A^TAx=A^Tb$. $\CG$ is a
deterministic algorithm that requires matrix-vector multiplications
with the entire data matrix $A$ and its transpose at every iteration,
while $\RK$ and $\ARK$ are randomized algorithms for which each
iteration requires access to just one row of $A$, but which require
many more iterations than $\CG$ in general. $\CG$ does not require
estimates of parameters such as $\lambdamin$ (though we show in the
next section that estimation of this parameter can be incorporated
into RK algorithms efficiently). Because the CG and RK approaches have
very different convergence properties, and because their data access
requirements are quite different, there are situations in which one or
other of them will have an advantage. Here we do a simple comparison
between CG and the RK methods based only on convergence rate as a
function of operation count, and put aside the issues of suitability
of one class or the other to various contexts and various
computational platforms.

The asymptotic convergence rate for $\CG$ is

\begin{equation} \label{eqn_cg_bound}
{\rcc
\|Ax_{k+1}-b\|^2 \leq
\left(\frac{\sqrt{\lambdamax}-\sqrt{\lambdamin}}{\sqrt{\lambdamax}+\sqrt{\lambdamin}}\right)^{2(k+1)}\|Ax_{0}-b\|^2.
}
\end{equation}
(See, for example, {\rcc formula
(5.36)} in \citep{NocedalWright06}.) The decrease factor per iteration
is thus approximately
\begin{equation}  \label{eqn_cg_rate}
1-{\rcc 4}{\sqrt{\lambdamin}\over \sqrt{\lambdamax}}.
\end{equation}
If $A$ has sparsity $\delta$, the cost of the main operation of $\CG$
--- multiplication by $A^TA$ --- is about $4 \delta mn$
operations. This is the approximate cost of $m$ iterations of $\RK$
and about $(2/3) \sqrt{\delta} m$ iterations of $\SARK$. Thus, for a
roughly equivalent number of operations, assuming again that $\mumin =
\lambdamin$, we obtain the following approximate decrease factors for
$\RK$ and $\SARK$:
\begin{subequations} \label{eq:rkcg}
\begin{alignat}{2} \label{eqn_rk_cg_rate}
\RK:&~\left(1-\frac{\lambdamin}{m} \right)^m  &&
\approx 1-\lambdamin,\\
\label{eqn_ark_cg_rate}
\ARK:&~\left(1-\frac{\sqrt{\lambdamin}}{m}\right)^{(2/3) \sqrt{\delta}m}
&& \approx 1-\frac23 \sqrt{\delta} \sqrt{\lambdamin}.
\end{alignat}
\end{subequations}

By comparing \eqref{eqn_cg_rate} and \eqref{eqn_rk_cg_rate}, we see
that $\RK$ may be competitive with $\CG$ if
$\sqrt{\lambdamin\lambdamax}$ (the geometrically averaged eigenvalue
of $A^TA$) is significantly larger than $1$.  From \eqref{eqn_cg_rate}
and \eqref{eqn_ark_cg_rate}, we see that $\SARK$ may be competitive
with $\CG$ if $\delta \lambdamax$ is significantly great than $1$.

We note however that the asymptotic rate \eqnok{eqn_cg_bound} for
$\CG$ is somewhat pessimistic. In practice, performance of $\CG$
depends on the distribution of the eigenvalues of $A^TA$. Rapid
convergence is often seen on early iterations, as the largest
eigenvalues are ``resolved,'' but the method often settles into a
steady linear rate on later iterations.

\section{Computational Results} \label{sec:computations}

In this section, we study the computational behavior of $\RK$, $\ARK$,
$\SARK$, and $\CG$ on a variety of test problems. We start by
comparing $\RK$ and $\ARK$ for dense $A$, then compare $\RK$, $\ARK$,
and $\SARK$ for sparse $A$. Finally, we compare the randomized
algorithms ($\RK$ and $\ARK$) to the deterministic algorithm
$\CG$.  

Since we need to supply the parameter $\mumin$ to $\ARK$, we introduce
three ways of setting this parameter:
\begin{itemize}
\item[-] $\ARK(\lambdamin)$: set $\mumin = \lambdamin$. This choice gives
  the theoretically best convergence rate, and should be used if
  $\lambdamin$ is known.
\item[-] $\ARK(0)$: set $\mumin = 0$. This choice requires no additional
  knowledge of $A$ and guarantees convergence, though at a sublinear
  rate (see \eqnok{eq:ark_sublin}).
\item[-] $\ARKauto$: $\mumin$ determined automatically. Run $\RK$ for
  $K_2$ iterations and record $x_{K_1+1}$ and $x_{K_2+1}$, where
  $K_2=\lceil {K\over 10}\rceil$ and $K_1=\max (1, K_2-10m)$. From
  \eqref{eqn_rk_bound}, we can say roughly that
  $\mathbb{E}(\|Ax_k-b\|^2) \sim (1-\lambdamin/m)^k$, so by setting
  $k=K_1$ and $k=K_2$, we deduce that $\lambdamin$ could be estimated
  by the formula
\begin{align*}
 m\left[1-\left({\|Ax_{K_2}-b\|\over \|Ax_{K_1}-b\|}\right)^{2\over K_2-K_1}\right].
\end{align*}
We find that a more conservative estimate of $\lambdamin$ works better
in practice, in which we
 replace the exponent $2/(K_2-K_2)$ by $0.5/(K_2-K_1)$ in our
 experiments.  If the entire matrix $A$ cannot be obtained at one
 time, one could estimate $\|Ax-b\|^2$ by using a sample of the rows
 of $Ax-b$.
\end{itemize}

We measure performance by plotting residual error $\|Ax-b\|$ against
the number of iterations and the number of operations
The initial point $x_0 = 0$ is used in all algorithms.

\subsection{Comparison between $\RK$ and $\ARK$ for Dense Data} \label{sec:NS_1}

Synthetic data for these tests is generated as follows: All elements
of the data matrix $A\in \R^{m\times n}$ and the optimal solution
$x^*\in \R^{n}$ are chosen to be i.i.d. $\mathcal{N}(0, 1)$. The
length of all rows in $A$ is normalized to $1$. The right-hand side
$b$ is set to $b=Ax^*$. We run all algorithms 20 times (with 20
different sample sequences) and report the averaged performance.

Figures~\ref{fig_e_1} and~\ref{fig_e_2} show residual errors for $\RK$
and $\ARK$ with different values of $\mumin$. Figure~\ref{fig_e_1}
focuses on small problems while Figure~\ref{fig_e_2} shows larger
cases. In the graphs in the left column, the horizontal axis is
iteration number, while in the right column, the horizontal axis is
operation count, which is our proxy for computation cost.  
{\rcc Operation count is obtained by scaling the number of iterations
  by our estimate of the average number of floating-point operations
  per iteration (see Table~\ref{tab_e}).} From these figures, we
observe the following.
\begin{itemize}
\item[-] $\ARK (\lambdamin)$ and $\ARKauto$ converge much faster than
  $\RK$ (in both iterations and operations), except for very well
  conditioned problems.
\item[-] After the initial phase in which $\lambdamin$ is estimated,
  $\ARKauto$ converges at about the same rate as $\ARK(\lambdamin)$.
\item[-] $\ARK(0)$ is not competitive with the other variants of $\ARK$,
  but is competitive with $\RK$ on ill conditioned problems.
\end{itemize}


\begin{figure}[htp!]
  \centering
    \subfloat{\includegraphics[width=0.8\textwidth]{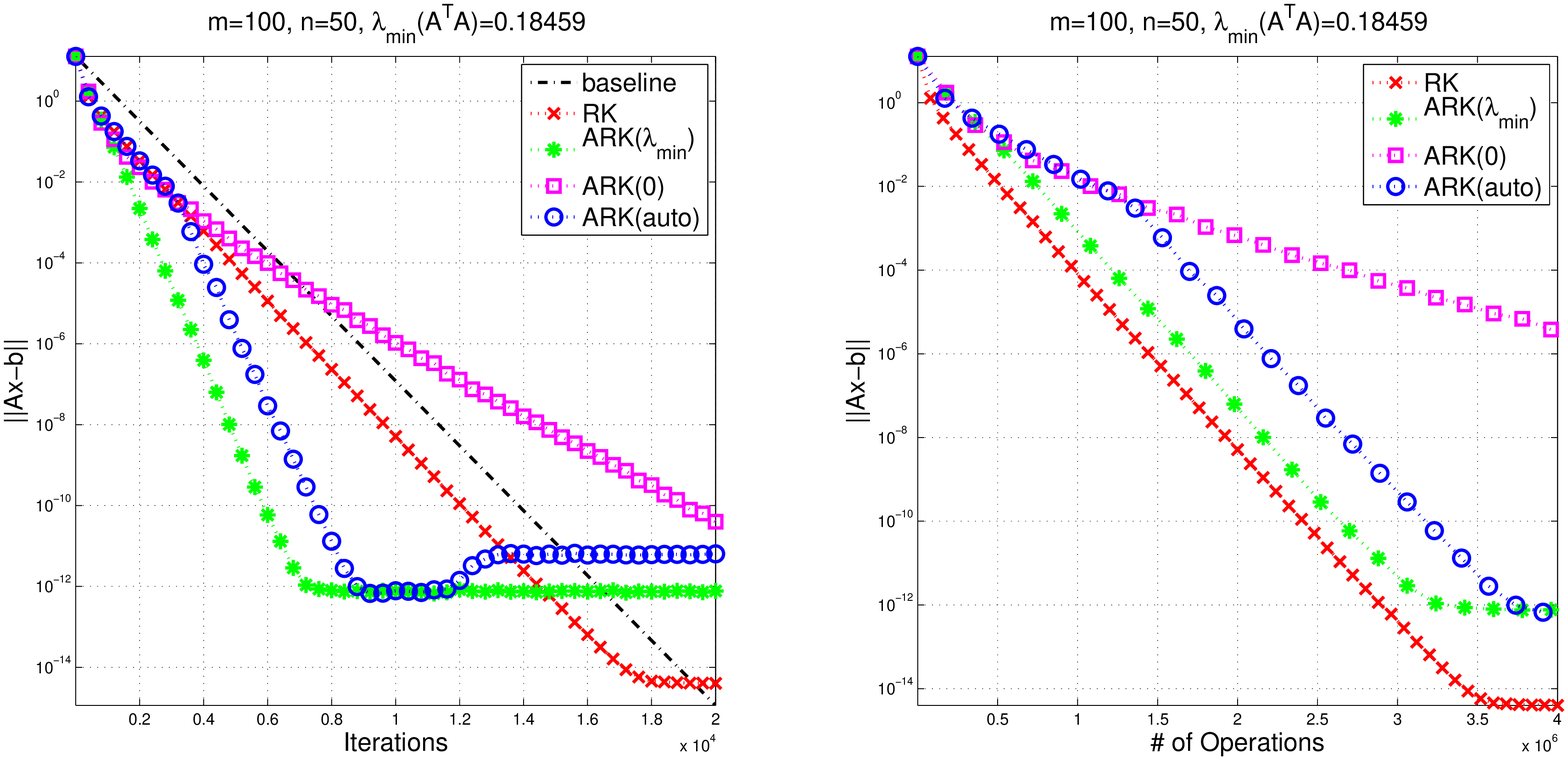}}\\
    \subfloat{\includegraphics[width=0.8\textwidth]{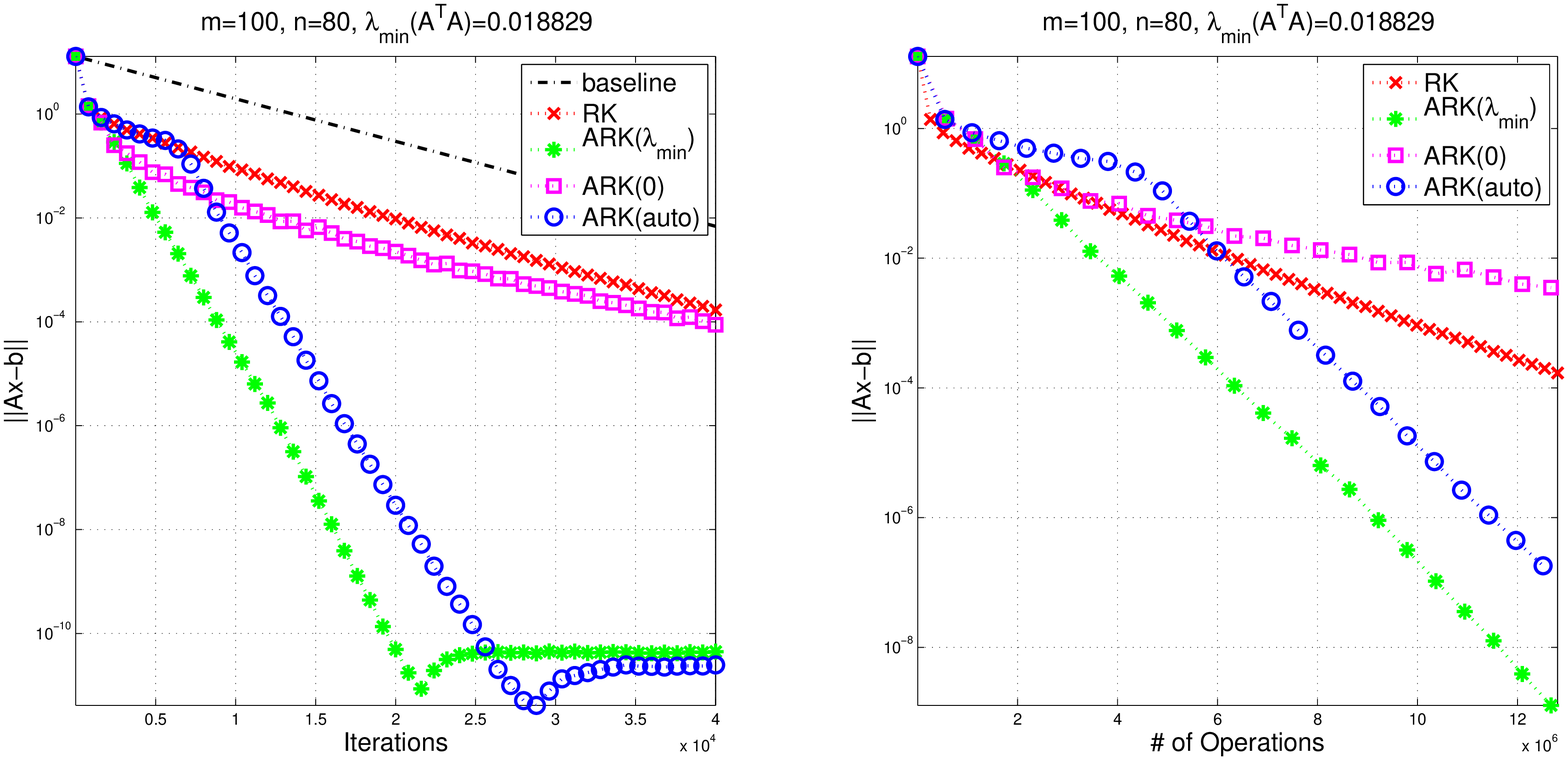}}\\
    \subfloat{\includegraphics[width=0.8\textwidth]{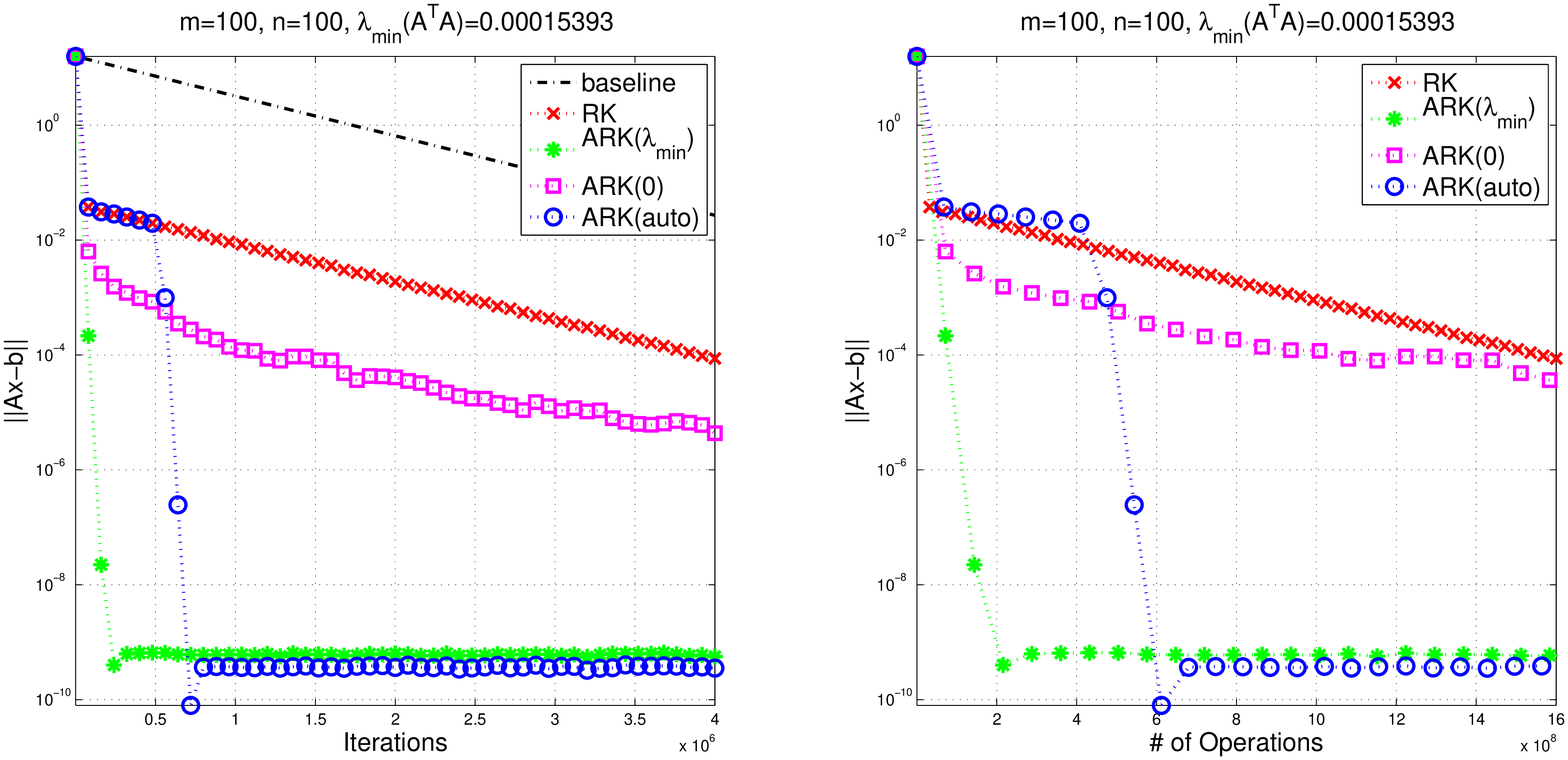}}
    \caption{Comparison among $\RK$, $\ARK(\lambdamin)$, $\ARK(0)$,
      and $\ARK(\mbox{\rm auto})$ on the dense data for $m=100$ and
      $n=50$, $80$, $100$. The graphs on the left (right) column plot
      iterations (operations) against residual error, averaged over 20
      trials. The left graphs show a reference baseline sequence
      $\{(1-\lambdamin/m)^k:~k=0,1,\dotsc\}$.} \label{fig_e_1}
\end{figure}
\begin{figure}[htp!]
  \centering
      \subfloat{\includegraphics[width=0.8\textwidth]{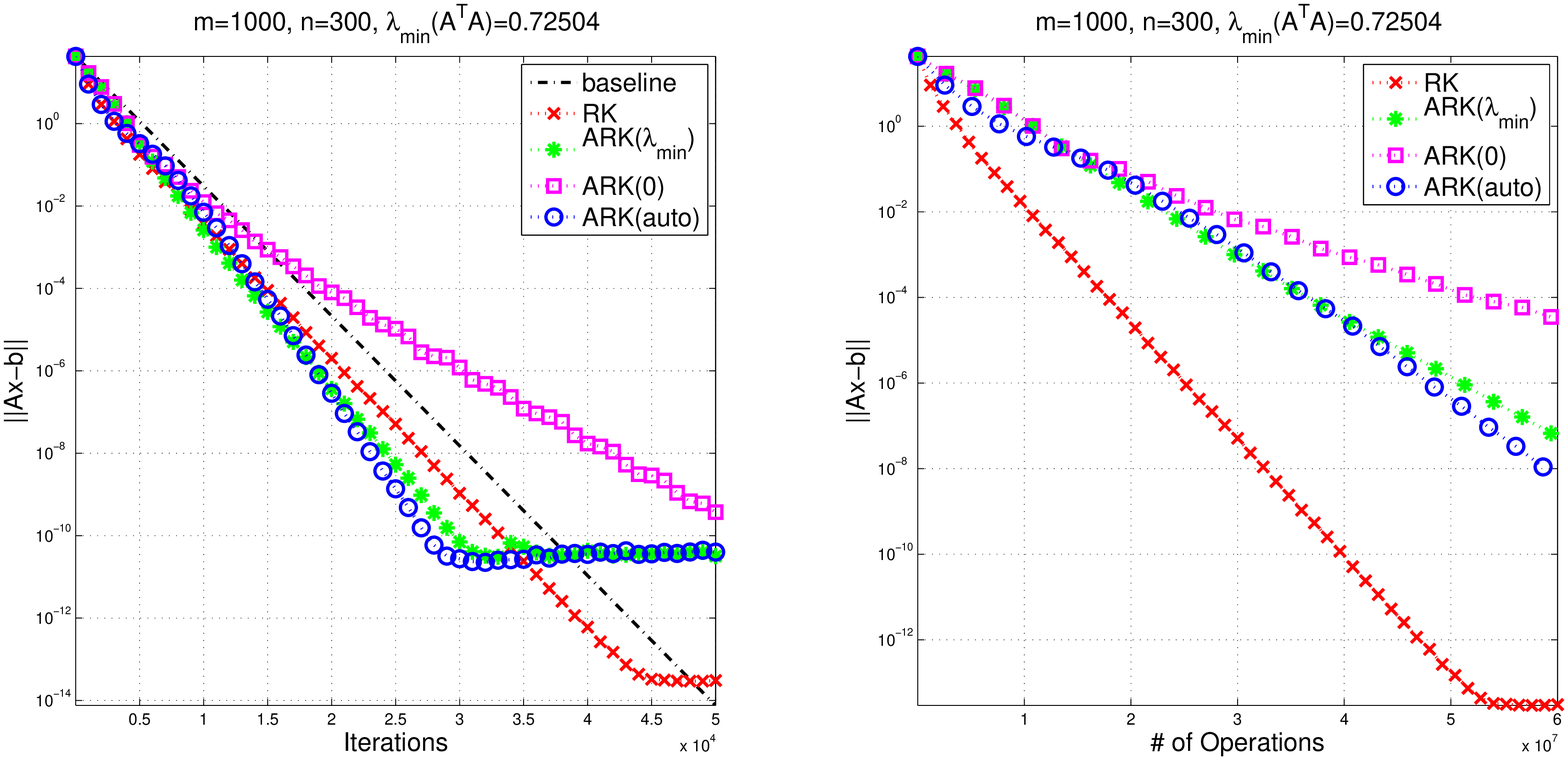}}\\
    \subfloat{\includegraphics[width=0.8\textwidth]{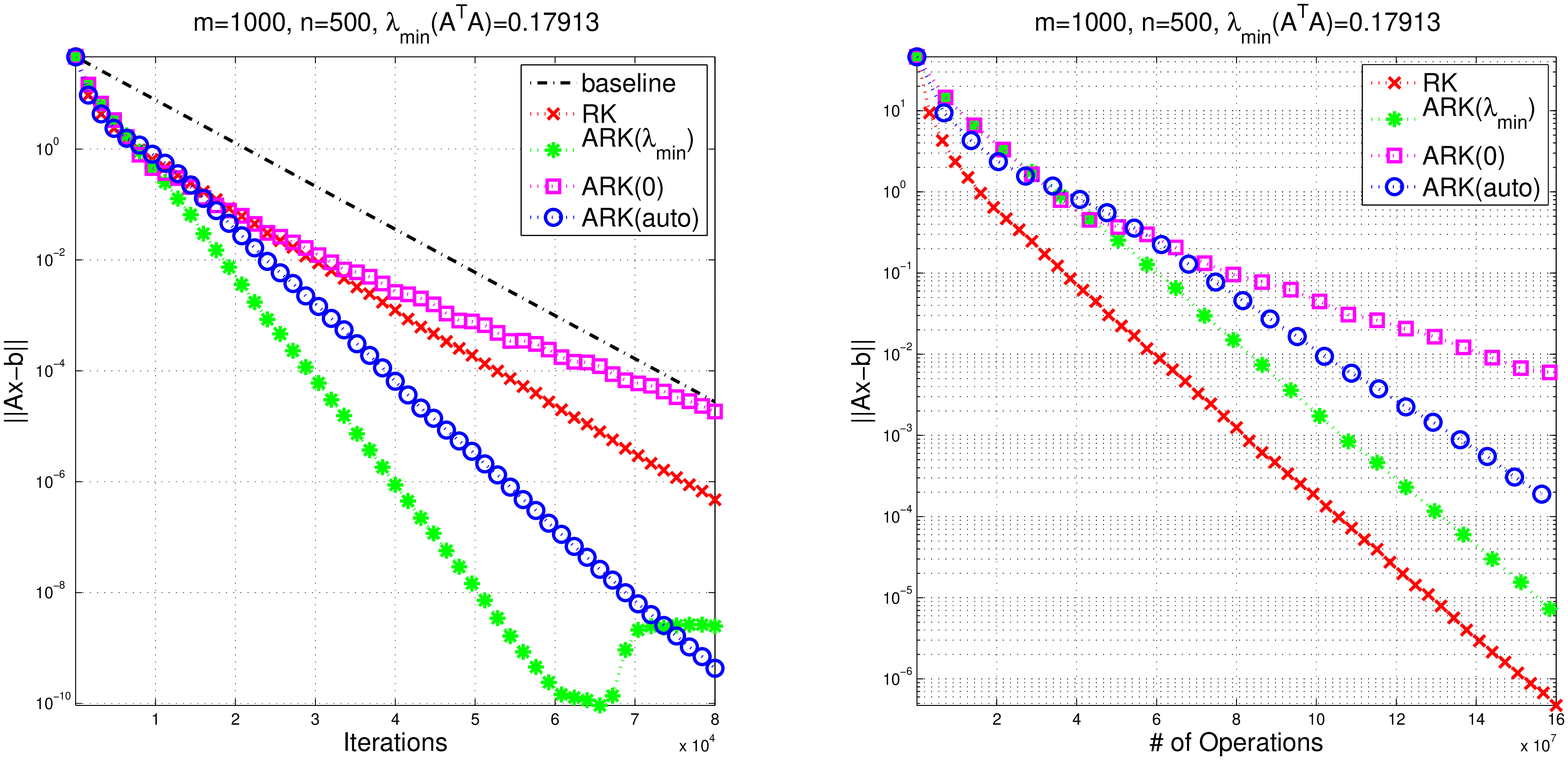}}\\
    \subfloat{\includegraphics[width=0.8\textwidth]{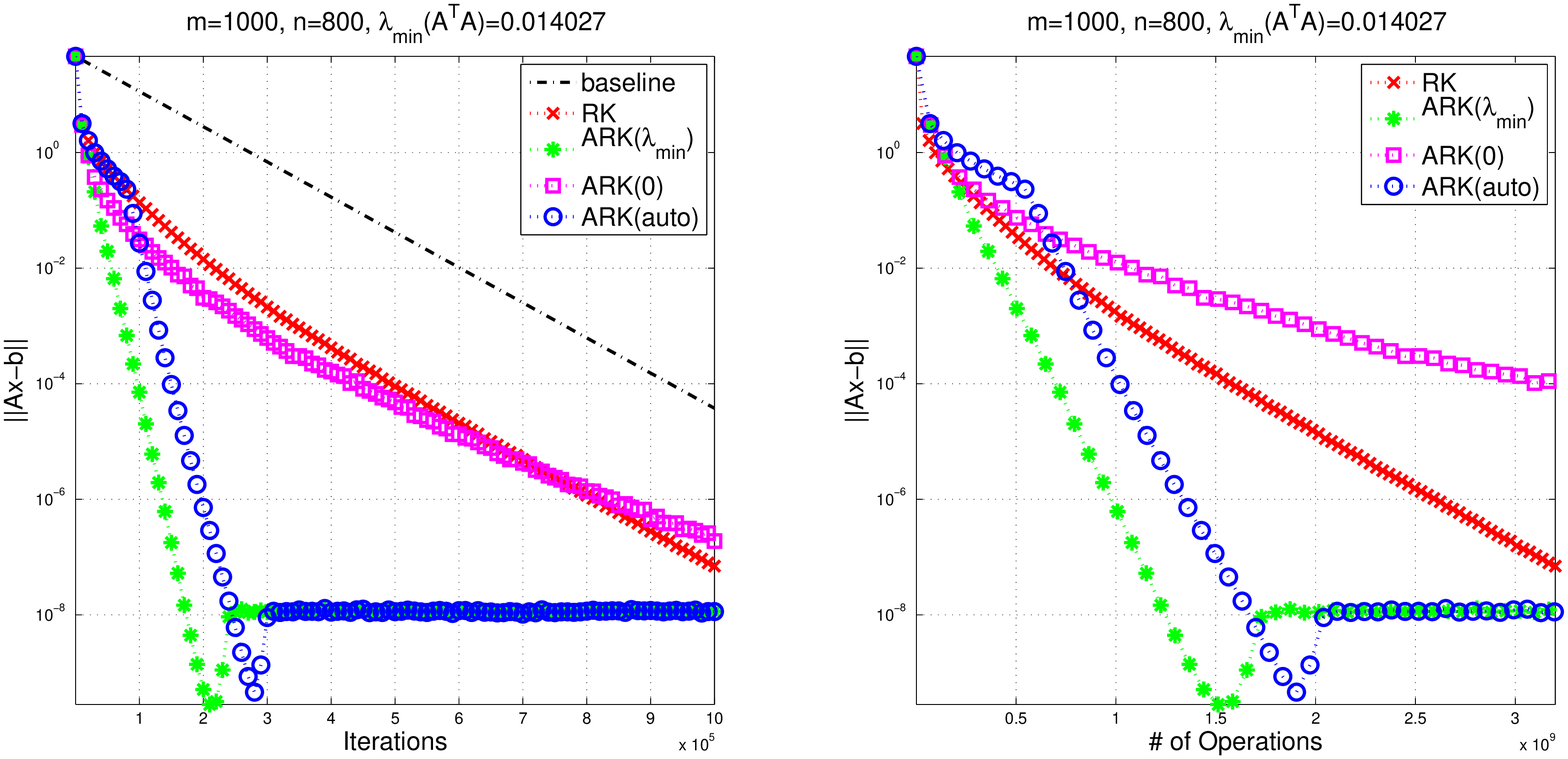}}
   \caption{Comparison among $\RK$, $\ARK(\lambdamin)$, $\ARK(0)$, and
     $\ARK(\mbox{\rm auto})$ on dense data for $m=1000$ and $n=300$,
     $500$, $800$. The graphs on the left (right) plot iterations
     (operations) against residual error, averaged over 20 trials. A
     reference baselien showing $\{(1-{\lambdamin/
       m})^k:~k=0,1,\dotsc\}$ is shown in the left
     plots.}\label{fig_e_2}
\end{figure}

\subsection{Comparison among $\RK$, $\ARK$, and $\SARK$ for Sparse Data}

We compare $\RK$, $\ARKauto$, and $\SARK$ on sparse data. Each element
of $A$ is set to $0$ with the probability $1-\delta$, so that the
proportion of nonzero entries in $A$ is approximately $\delta$. The
nonzero entries are chosen to be {\rcc i.i.d. Gaussian $\mathcal{N}(0,1)$,
then the zeros rows are removed from $A$ and the nonzero rows are
normalized.} The optimal solution $x^*$ and right-hand side $b$ are
generated as in the dense case.

Figure~\ref{fig_es_1} fixes $m=1000$ and $n=950$,
and chooses $\delta=0.8$, $0.08$, and $0.01$ for different levels of
sparsity. For this small value of $\lambdamin$ (about $.0006$ in all
three cases), $\ARK/\SARK$ outperforms $\RK$ with respect to number of
iterations, {\rcc as we see in the graphs in the left column of
  Figure~\ref{fig_es_1}}. For the highest density ($\delta=0.8$; top
right graph), both $\ARK$ and $\SARK$ take fewer operations than
$\RK$, and $\ARK$ is more efficient than $\SARK$. For moderate
sparsity $\delta = 0.08$ (middle right), $\ARK$ is dominated by $\RK$
in operation count, while $\SARK$ is the best option of the three. For
the most sparse case ($\delta=0.01$; bottom right), $\RK$ dominates
both $\ARK$ and $\SARK$ in the number of operations. These
observations are consistent with our analysis of
Section~\ref{sec:equality_cmp_sparse}.

\begin{figure}[htp!]
\centering
       \subfloat{\includegraphics[width=0.8\textwidth]{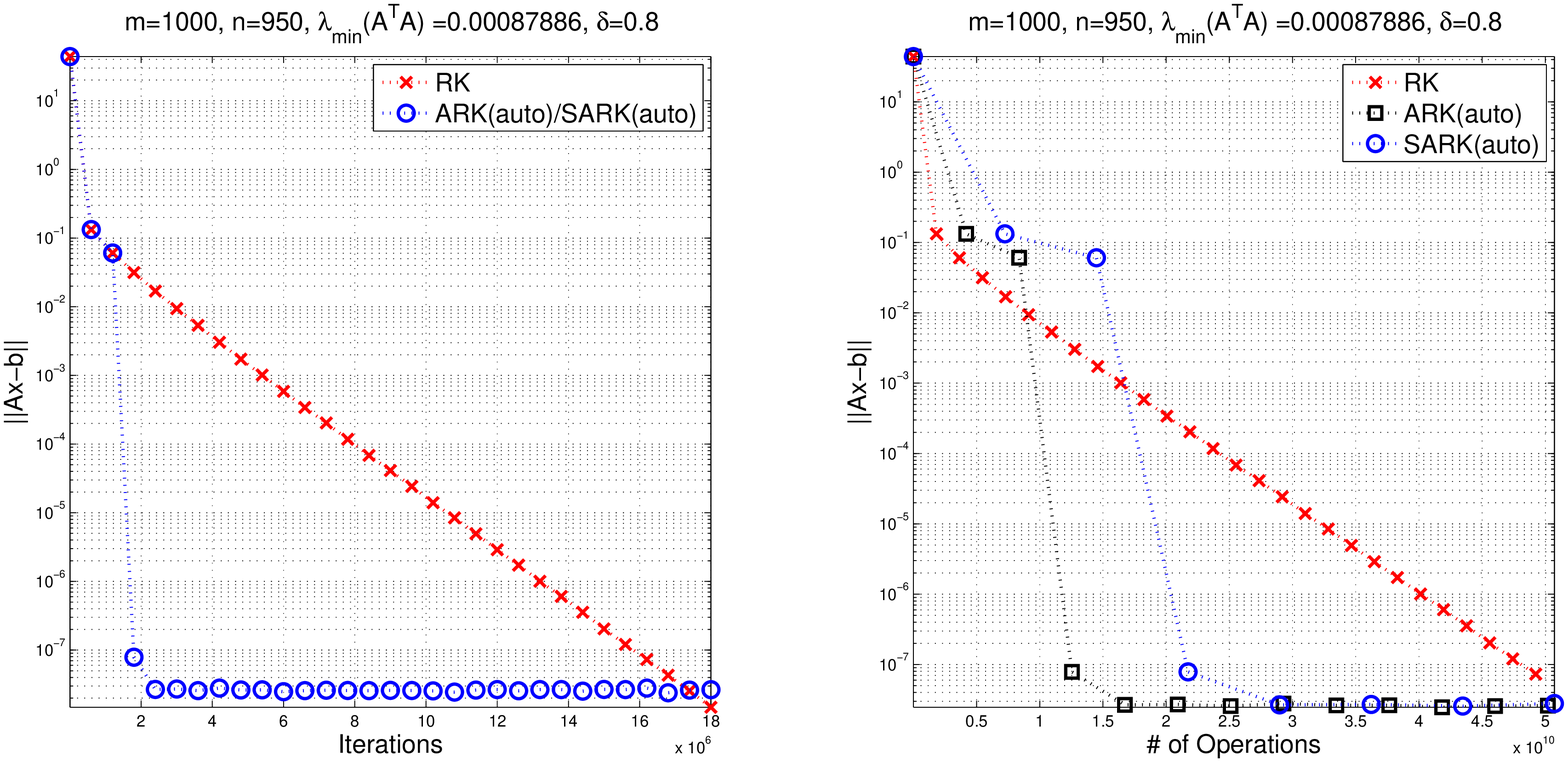}}\\
       \subfloat{\includegraphics[width=0.8\textwidth]{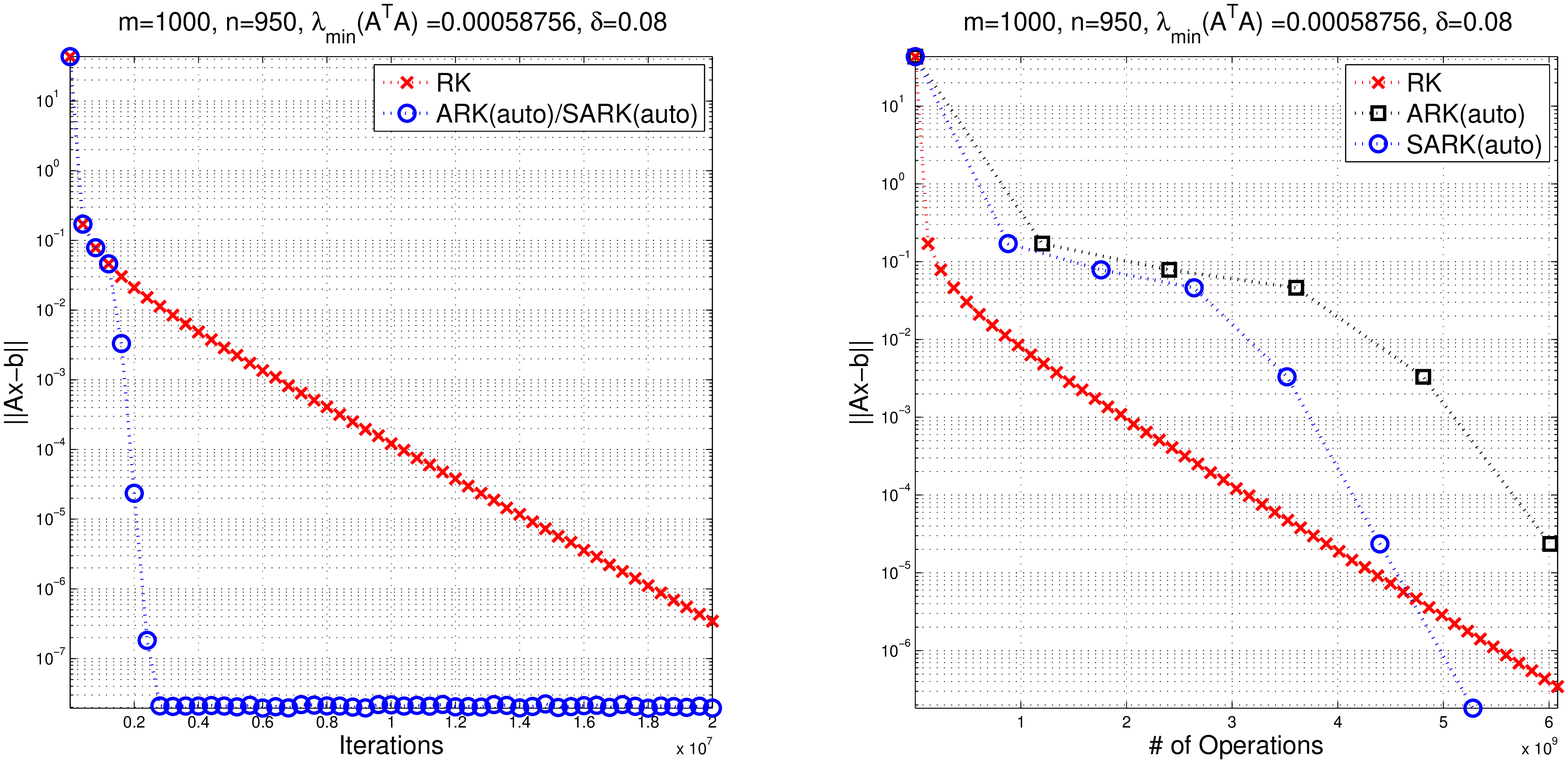}}\\
    \subfloat{\includegraphics[width=0.8\textwidth]{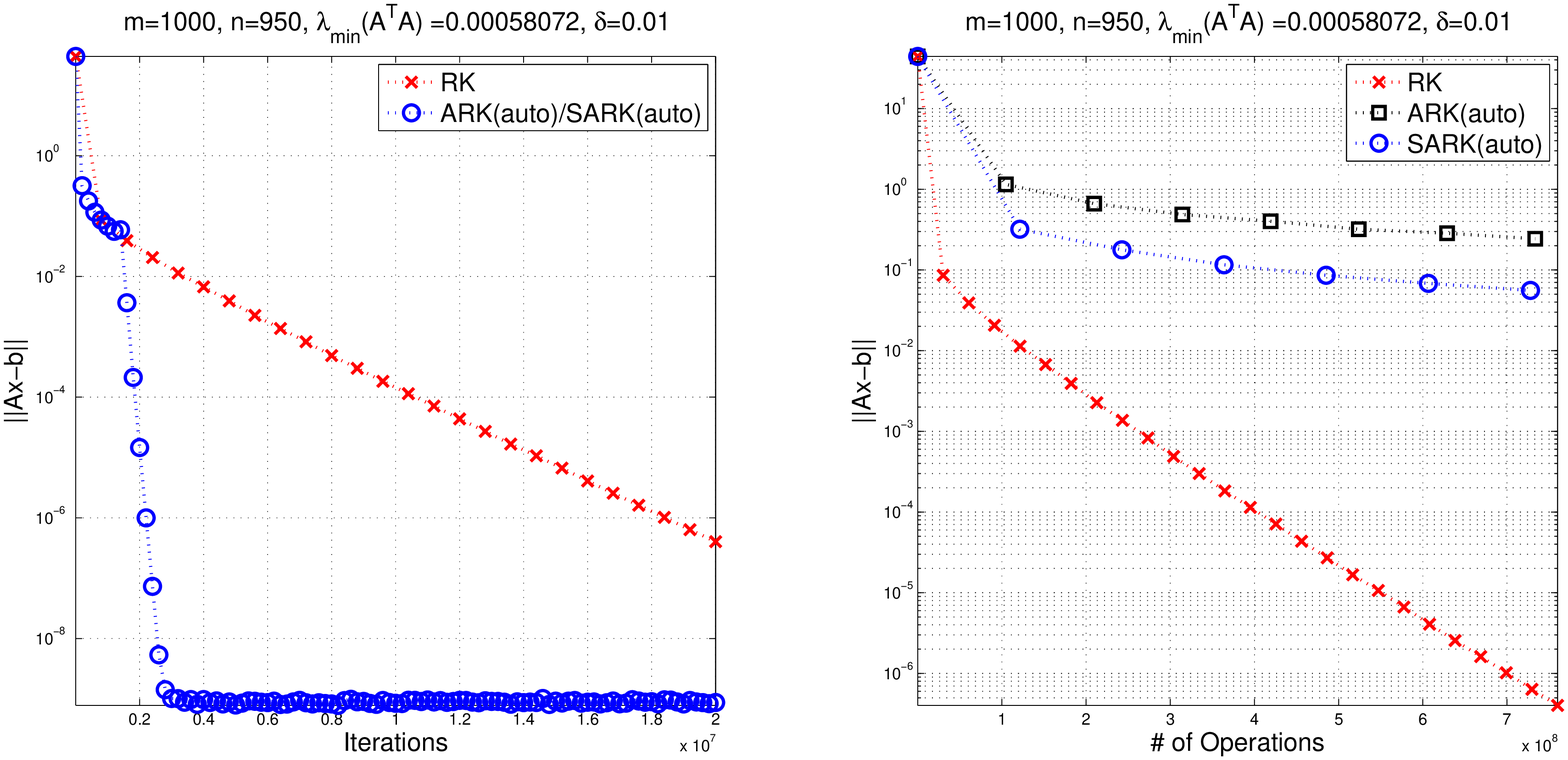}}
\caption{Comparison among $\RK$, $\ARK(0)$, and $\SARKauto$ on sparse
  data with $m=1000$, $n=950$, and $\delta = 0.01$, $0.08$, and
  $0.8$. The graphs on the left (right) column plot iterations
  (operations) against residual errors, averaged over 20 trials.}
\label{fig_es_1}
\end{figure}

\subsection{Comparison among $\RK$, $\ARK$, and $\CG$}


{\rcc
A comparison between $\CG$ and $\RK$ was made in \cite{Strohmer09},
where $A$ is chosen to be Gaussian (elements are i.i.d. from
$N(0,1/n)$) with $m \gg n$. Problems of this type are particularly
advantageous for $\RK$. From random matrix theory \citep{Vershynin11},
we have for these matrices that $\lambdamin\approx (\sqrt{m/n}-1)^2$
and $\lambdamax\approx (\sqrt{m/n}+1)^2$, so that when $m\gg n$, we
have $\sqrt{\lambdamin\lambdamax}\gg 1$.} The convergence rates
observed in \citep{Strohmer09} are thus consistent with our analysis
of Section~\ref{sec:boundcmp_cg}. We do not consider the case $m\gg n$
further here, because $\lambdamin$ is large in this setting, so all
algorithms converge rapidly. We focus instead on cases in which $m=n$
and $A$ is ill conditioned.

For a given choice of $\lambdamin$, we see from
Section~\ref{sec:boundcmp_cg} that $\CG$ favors a smaller {\em
  maximum} eigenvalue, while $\RK$ and $\ARK$ favor {\rcc a smaller {\em
  geometric average} eigenvalue}. We control the distribution of
eigenvalues of $A^TA$ by generating our test matrices as
follows. First, find the SVD $U \Lambda V^T$ of a random $n \times n$
Gaussian matrix. Next, define an $n \times n$ diagonal matrix
$\tilde{\Lambda}$ by $\tilde{\Lambda}_{ii}=i^{-\alpha}$,
$i=1,2,\dotsc,n$, for some parameter $\alpha>0$, and compute
$U\tilde{\Lambda} V^T$. Finally, normalize the rows of this matrix to
obtain $A$. We generate $x^*$ and $b$ in the same way as in
Section~\ref{sec:NS_1}. The rows of $A$ are normalized, so $\trace
(A^TA)=n$ and the average eigenvalue of $A^TA$ is $1$. The parameter
$\alpha$ controls the distribution of eigenvalues of $A^TA$; as
$\alpha$ increases, $\lambdamax$ tends to grow while $\lambdamin$
shrinks.

We choose three values of $\alpha$ --- $0.5$, $0.75$, and $0.9$ ---
and fix $n=500$ in Figure~\ref{fig_e_CG_1}. {\rcc Each row of plots in
Figure~\ref{fig_e_CG_1} corresponds to a particular value of $\alpha$,
increasing from top to bottom. The left column plots the number of
iterations of each method, but since the complexity of $\CG$ per
iteration is $O(n^2)$ while that of other algorithms is $O(n)$, we do
a rough calibration by making each iteration of $\CG$ occupy $n$ units
on the horizontal axes of the graphs in this column. }
{\rcc We note that $\CG$ converges rapidly in its early iterations but
  then slows.} This behavior is consistent with the analysis of $\CG$,
which shows that the asymptotic rate \eqnok{eqn_cg_rate} is somewhat
pessimistic, and that early iterations tend to behave in a manner
dictated by the distribution of eigenvalues of $A^TA$ rather than the
ratio of the extreme eigenvalues. Rapid initial convergence is enabled
by the fact that each iteration of $\CG$ does a sweep over the entire
matrix, giving it a global view of the data which is lacking in the
randomized approaches.  By contrast with $\CG$ the convergence of
randomized algorithms is consistent and stable, and well predicted by
the analysis.

{\rcc
As the value of $\alpha$ increases (that is, as we move from the top
row of plots to the bottom row in Figure~\ref{fig_e_CG_1}), we observe
the following changes.}
\begin{itemize}
\item[-] $\lambdamin$ becomes smaller, $\lambdamax$ becomes larger, and
  $\sqrt{\lambdamin \lambdamax}$ becomes smaller, as $\alpha$
  increases.
\item[-] The asymptotic convergence rate of $\CG$, after
  resolution of the leading eigenspaces, becomes slower as $\alpha$
  increases.
\item[-] The performance of $\RK$ becomes worse compared to
  $\CG$ as $\alpha$ grows. This observation is consistent
  with our analysis in Section~\ref{sec:boundcmp_cg}, which predicts
  poorer performance as $\sqrt{\lambdamin \lambdamax}$ decreases.
\item[-] {\rcc The performance of $\ARK$ (including $\ARK(\mumin)$ and
  $\ARK(\mbox{\rm auto})$) is comparable to $\CG$. $\CG$ decreases
  faster in the beginning but $\ARK$ is better at achieving high
  precision. An effective hybrid strategy might be to run $\CG$ in
  early iterations and turn to $\RK$ or $\ARK$ in later iterations.}
\end{itemize}

\begin{figure}[htp!]
\centering
       \subfloat{\includegraphics[width=0.8\textwidth]{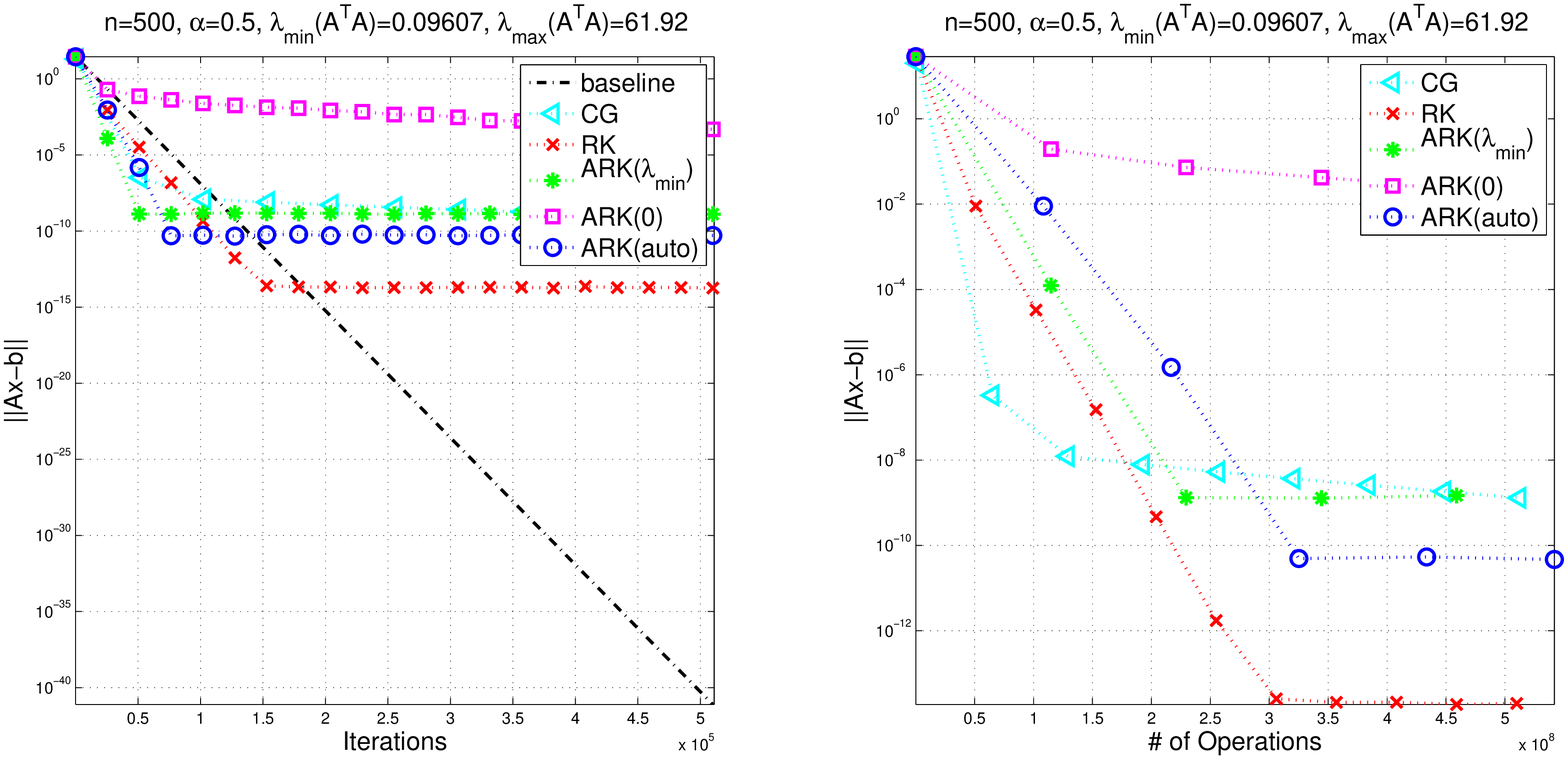}}\\
        \subfloat{\includegraphics[width=0.8\textwidth]{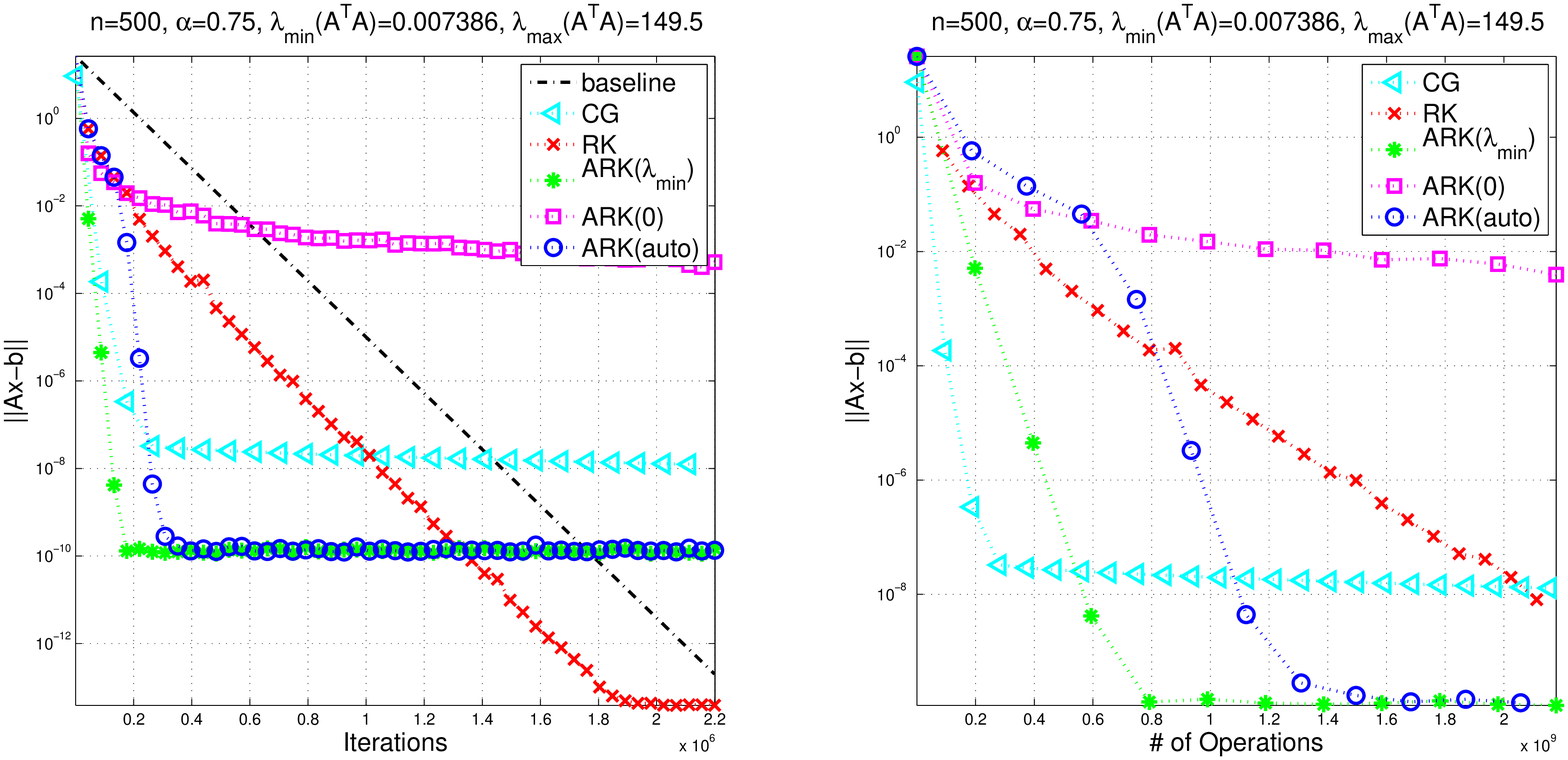}}\\
    \subfloat{\includegraphics[width=0.8\textwidth]{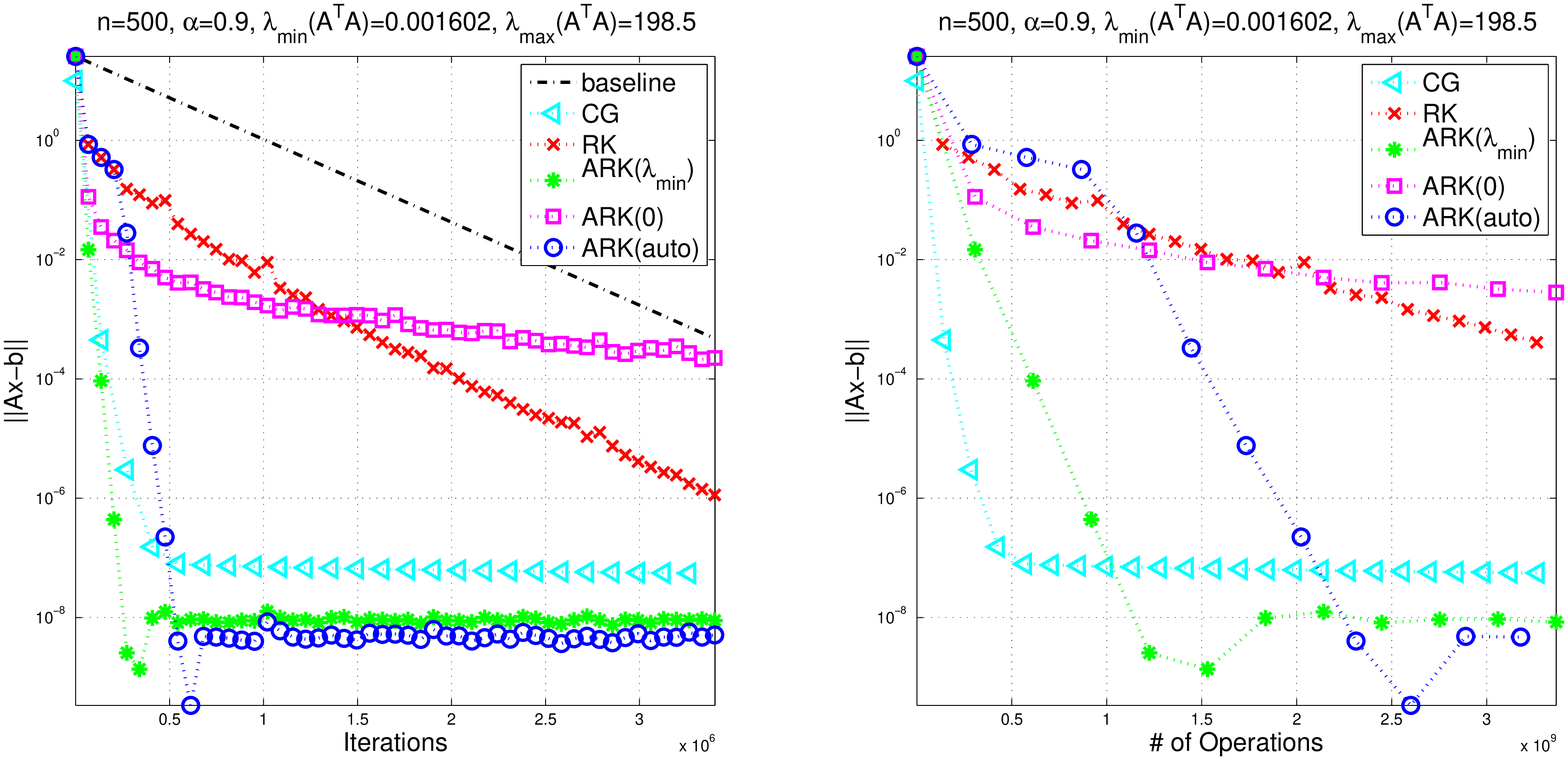}}
        \caption{Comparison amomg $\CG$, $\RK$,
          $\ARK(\lambdamin)$, $\ARK(0)$, and $\ARK(\mbox{\rm auto})$
          on dense data. The figures on the left (right) plot residual
          against iterations (operations). A reference baseline
          sequence of $\{(1-{\lambdamin / m})^k:~k=0,1,2,\dotsc\}$ is
          shown in the left plots.}
\label{fig_e_CG_1}
\end{figure}

{
\bibliographystyle{amsplain}
\bibliography{reference} 
}

\newpage

\appendix

\section{Proof of Theorem~\ref{thm_main}}

In proving Theorem~\ref{thm_main}, we refer to the particular
implementation in Algorithm~\ref{alg_arka} of $\ARK$. We assume
throughout that $\| a_i \|_2=1$ for $i=1,2,\dotsc,m$.


We start with two useful technical lemmas.
\begin{lemma}  \label{lem_1}
For any $y\in\R^{n}$, we have
\begin{equation}
\mathbb{E}_{i}\left(\left\|{a_i(a_i^Ty-b_i)}\right\|^2_{(A^TA)^+}\right)
\leq {1\over m}\|Ay-b\|^2,
\end{equation}
where the random variable $i$ follows the uniform distribution over
the set $\{1,2,\dotsc,m\}$.
\end{lemma}
\begin{proof}
Define the compact singular value decomposition of $A$ as $A=U\Sigma
V^T$, where $U^TU=I$, $V^TV=I$, and $\Sigma$ is positive diagonal, so
that $(A^TA)^+ = V \Sigma^{-2} V^T$.  Denoting $U^T = [u_1 \, u_2 \,
  \dotsc \, u_m]$, it is easy to show that $\|u_i \|_2 \le 1$ for all
$i=1,2,\dotsc,m$. Using $\mathbb{E}_i$ to denote expectation with
respect to the index $i$, we have
\begin{align*}
&\mathbb{E}_{i}\left(\left\|{a_i(a_i^Ty-b_i)}\right\|^2_{(A^TA)^+}\right)\\
&={1\over m}\sum_{i=1}^m\langle (A^TA)^+a_i(a_i^Ty-b_i),~a_i(a_i^Ty-b_i) \rangle\\
&={1\over m} \trace \left[(A^TA)^+\sum_{i=1}^ma_i(a_i^Ty-b_i)^2a_i^T\right]\\
&={1\over m} \trace \left[(A^TA)^+A^T \diag (Ay-b)^2A\right]\\
&={1\over m} \trace \left[V \Sigma^{-1} U^T  \diag (Ay-b)^2 U \Sigma V^T\right]\\
&={1\over m} \trace \left[ U^T  \diag (Ay-b)^2 U \right]\\
&={1\over m}\| \diag (Ay-b)U\|_F^2\\
&={1\over m}\sum_{i=1}^m(a_i^Ty-b)^2\|u_i\|^2\\
&\leq {1\over m}\|Ay-b\|^2.
\end{align*}
\end{proof}

\begin{lemma}
For any solution $x^*$ to \eqref{eqn_problem} and
any $y\in \R^n$, we have
\begin{equation}
\mathbb{E}_{i}(\|\mathcal{P}_{a_i,b_i}(y)-x^*\|^2)=\|y-x^*\|^2-{1\over
m}\|Ay-b\|^2,
\end{equation}
where the random variable $i$ follows the uniform distribution over
the set $\{1,2,\dotsc,m\}$. \label{lem_2}
\end{lemma}

\begin{proof}
We have
\begin{align*}
&\mathbb{E}_{i}(\|\mathcal{P}_{a_i,b_i}(y)-x^*\|^2)\\
&=\mathbb{E}_{i}\left(\left\|y-{a_i(a_i^Ty-b_i)}-x^*\right\|^2\right)\\
&=\|y-x^*\|^2 +
\mathbb{E}_{i}\left({\|a_i^Ty-b_i\|^2}\right)-2\left\langle
y-x^*,~\mathbb{E}_{i}\left({a_i(a_i^Ty-b_i)}\right) \right\rangle\\
&=\|y-x^*\|^2 + {1\over m}\|Ay-b\|^2 - {2\over m}\left\langle
A(y-x^*),
Ay-b \right\rangle\\
&=\|y-x^*\|^2-{1\over m}\|Ay-b\|^2,
\end{align*}
where the  last equality uses $Ax^*=b$.
\end{proof}

The proof of Theorem~\ref{thm_main} below essentially follows the
proof for accelerated coordinate descent algorithm in
\cite{Nesterov10} to construct the key inequality
\eqref{eqn_proof_key}.

\begin{proof}
From Algorithm~\ref{alg_arka} one can verify that if the sequence
$\{x_k, y_k, v_k\}$ is generated from $\ARK(A, b, \mumin, x_0, K)$,
then the sequence generated from $\ARK(A, b-Ax_0, \mumin, 0, K)$ must
be $\{x_k-x_0, y_k-x_0, v_k-x_0\}$.  Thus, solving $Ax=b$ is
equivalent to solving $Ax=b-Ax_0$ from initial point $0$. It therefore
suffices to study convergence from the zero initial point.

Recall from \eqref{eqn_alg1} that $\gamma_k$ is the larger root of the
following convex quadratic function:
\[
t(\gamma) := \gamma^2 - \frac{\gamma }{m} (1-\mumin \gamma_{k-1}^2) -
\gamma_{k-1}^2.
\]
Since $\mumin \le \lambdamin \le m$, and using $\gamma_{-1}=0$, we can
note the following, from a simple recursive argument:
\[
t(0) = -\gamma_{k-1}^2  \le 0, \quad
t(1/m) = \gamma_{k-1}^2 (\mumin/m^2-1) \le 0,
\]
and thus $\gamma_k \ge 1/m$ for all $k \ge 0$. We can also verify that
if $\gamma_{k-1} \le 1/ \sqrt{\mumin}$, we have
\begin{align*}
t(\gamma_{k-1}) &= -(\gamma_{k-1}/m) (1-\mumin \gamma_{k-1}^2)  \le 0 \\
t \left( \frac{1}{\sqrt{\mumin}} \right)
& = \frac{1}{\mumin} - \frac{1}{m \sqrt{\mumin}}
(1-\mumin \gamma_{k-1}^2) - \gamma_{k-1}^2 \\
& = \frac{1}{\mumin}  - \frac{1}{m \sqrt{\mumin}} +
\gamma_{k-1}^2 \left( \frac{\sqrt{\mumin}}{m}-1\right) \\
& \ge \frac{1}{\mumin} - \frac{1}{m \sqrt{\mumin}} +
\frac{1}{\mumin} \left( \frac{\sqrt{\mumin}}{m}-1\right) =0,
\end{align*}
which together imply that
\[
\gamma_k \in \left[ \gamma_{k-1}, \frac{1}{\sqrt{\mumin}} \right].
\]
It follows from these bound (together with the initialization
$\gamma_{-1}=0$) that $\{ \gamma_k \}_{k=0}^{\infty}$ is an increasing
sequence, bounded below by $1/m$ and above by $1/\sqrt{\mumin}$. It
follows from these bounds and from $\mumin \le m$ that $\alpha_k$ and
$\beta_k$ both lie in the interval $[0,1]$ for all $k$.


Recalling that $x_0=0$, we have $x^* = A^+b$. It can be verified that
$x_k$, $y_k$, $v_k$, and $x^*$ are all in $\mathcal{R}(A^T)$. 
We observe some useful relationships among the scalars in the
algorithm.  We have from \eqref{eqn_alg1} and \eqref{eqn_alg2} that
\begin{equation} \label{eqn_r21}
\frac{1-\alpha_k}{\alpha_k} = \frac{m^2 \gamma_k - m}{m-\gamma_k \mumin}
= \frac{m}{\gamma_k} \frac{m \gamma_k^2 - \gamma_k}{m-\gamma_k \mumin} =
\frac{m \gamma_{k-1}^2}{\gamma_k} .
\end{equation}
From \eqref{eqn_alg1} and \eqref{eqn_alg3}, we have
\begin{equation} \label{eqn_r20}
\gamma_k^2-{\gamma_k\over m} - \beta_{k}\gamma_{k-1}^2=0.
\end{equation}


Defining
\begin{equation} \label{eqn_defr}
r_k:=\|v_k-x^*\|_{(A^TA)^+},
\end{equation}
we consider the following expansion of $r_{k+1}^2$.
\begin{align}
\nonumber
  r_{k+1}^2 &=\|v_{k+1}-x^*\|_{(A^TA)^+}^2\\
\nonumber
  &=\left\|\beta_kv_k+(1-\beta_k)y_k-\gamma_k{a_i(a_i^Ty_k-b_i)}-x^*\right\|_{(A^TA)^+}^2\\
\nonumber
  &=\|\beta_kv_k+(1-\beta_k)y_k-x^*\|^2_{(A^TA)^+} +
  \gamma_k^2\left\|{a_i(a_i^Ty_k-b_i)}\right\|^2_{(A^TA)^+} \\
\nonumber
  & \quad -2\gamma_k\left\langle
  \beta_kv_k+(1-\beta_k)y_k-x^*,~(A^TA)^+a_i(a_i^Ty_k-b_i)\right\rangle\\
\nonumber
  &=\|\beta_kv_k+(1-\beta_k)y_k-x^*\|^2_{(A^TA)^+} +
  \gamma_k^2\left\|{a_i(a_i^Ty_k-b_i)}\right\|^2_{(A^TA)^+} \\
\nonumber
  & \quad -2\gamma_k\left\langle
  \beta_k\left({1\over \alpha_k}y_k-{1-\alpha_k\over \alpha_k}x_k\right)+(1-\beta_k)y_k-x^*,~{(A^TA)^+a_i(a_i^Ty_k-b_i)}\right\rangle\\
\nonumber
 &{\rcc =} \|\beta_kv_k+(1-\beta_k)y_k-x^*\|^2_{(A^TA)^+} +
  \gamma_k^2\left\|{a_i(a_i^Ty_k-b_i)}\right\|^2_{(A^TA)^+} \\
\label{eqn_r}
  & \quad +2\gamma_k\left\langle
   x^*-y_k+{1-\alpha_k\over \alpha_k}\beta_k(x_k-y_k),~{(A^TA)^+a_i(a_i^Ty_k-b_i)}\right\rangle.
\end{align}
Denote by $i(k)$ the index randomly generated at iteration $k$, and
let $I(k)$ denote all random indices seen at or before iteration $k$,
that is,
\[
I(k):=\{i(k), i(k-1), \dotsc, i(0)\}.
\]
Note that $x_{k+1}$, $y_{k+1}$, and $v_{k+1}$ are determined by
$I(k)$. In the remainder of the proof, we use
$\mathbb{E}_{i(k)|I(k-1)}(\cdot)$ to denote the expectation of a
random variable with respect to the index $i(k)$, conditioned on
$I(k-1)$. Note that $\mathbb{E}_{I(k)} (\cdot) = E_{I(k-1)} (
E_{i(k)|I(k-1)} (\cdot))$. When the context is clear, we use $i$ in
place of $i(k)$.

We consider the three terms in \eqref{eqn_r} in turn. From the
convexity of $\|.\|^2_{(A^TA)^+}$ and the definition of $\beta_k$, the
first item can be bounded as follows:
\begin{align}
\nonumber
&\|\beta_kv_k+(1-\beta_k)y_k-x^*\|^2_{(A^TA)^+}\\
\nonumber
& \leq  \beta_k\|v_k-x^*\|^2_{(A^TA)^+} + (1-\beta_k)\|y_k-x^*\|^2_{(A^TA)^+} \\
\nonumber
& =  \beta_k\|v_k-x^*\|^2_{(A^TA)^+} + {\gamma_k\mumin\over m}\|y_k-x^*\|^2_{(A^TA)^+}\\
\label{eqn_r1}
& \leq \beta_k\|v_k-x^*\|^2_{(A^TA)^+} + {\gamma_k\over m}\|y_k-x^*\|^2,
\end{align}
where the last inequality is a consequence of $\mumin\leq
\lambdamin$ {\rcc and the fact that $y_k$ and $x^*$ are in $\mathcal{R}(A^T)$}. Using Lemmas~\ref{lem_1} and \ref{lem_2}, the
second item in \eqref{eqn_r} can be bounded in the expectation sense
as follows:
\begin{align}
\nonumber
 \mathbb{E}_{i(k)|I(k-1)} & \left(\left\|{a_i(a_i^Ty_k-b_i)}\right\|^2_{(A^TA)^+}\right) \\
& \leq {1\over m}\|Ay_k-b\|^2
\label{eqn_r2}
= \|y_k-x^*\|^2-\mathbb{E}_{i(k)|I(k-1)}(\|x_{k+1}-x^*\|^2).
\end{align}
For the third term in \eqref{eqn_r}, we have by taking an expectation that
\begin{align}
\nonumber
 &\mathbb{E}_{i(k)|I(k-1)}\left\langle x^*-y_k+{1-\alpha_k\over \alpha_k}\beta_k(x_k-y_k),~(A^TA)^+\left(a_i(a_i^Ty_k-b_i)\right) \right\rangle \\
\nonumber
& =  \left\langle x^*-y_k+{1-\alpha_k\over \alpha_k}\beta_k(x_k-y_k),~(A^TA)^+\mathbb{E}_{i(k)|I(k-1)}\left(a_i(a_i^Ty_k-b_i)\right) \right\rangle \\
\nonumber
& = {1\over m}\left\langle x^*-y_k+{1-\alpha_k\over \alpha_k}\beta_k(x_k-y_k),~(A^TA)^+\sum_{i}a_i(a_i^Ty_k-b_i) \right\rangle \\
\nonumber
& =  {1\over m}\left\langle x^*-y_k+{1-\alpha_k\over \alpha_k}\beta_k(x_k-y_k),~(A^TA)^+A^TA(y_k-x^*) \right\rangle \\
\nonumber
& \le {1\over m}\left\langle x^*-y_k+{1-\alpha_k\over \alpha_k}\beta_k(x_k-y_k),~y_k-x^* \right\rangle \\
\nonumber
& = {1\over m}\left(-\|y_k-x^*\|^2 + {1-\alpha_k\over \alpha_k}\beta_k\left\langle x_k-y_k,~y_k-x^* \right\rangle\right)\\
\nonumber
& = {1\over m}\left(-\|y_k-x^*\|^2 + {1-\alpha_k\over 2\alpha_k}\beta_k\left(\|x_k-x^*\|^2 -\|y_k-x^*\|^2 - \|x_k-y_k\|^2\right)\right)\\
\nonumber
& = {1\over m}\left(-\left(1+{1-\alpha_k\over 2\alpha_k}\beta_k\right)\|y_k-x^*\|^2 + {1-\alpha_k\over 2\alpha_k}\beta_k\left(\|x_k-x^*\|^2 -\|x_k-y_k\|^2\right)\right)\\
\nonumber
& = -\left({1\over m}+{\beta_k\gamma^2_{k-1}\over 2\gamma_k}\right)\|y_k-x^*\|^2 + {\beta_k\gamma_{k-1}^2\over 2\gamma_k}\left(\|x_k-x^*\|^2 -\|x_k-y_k\|^2\right)
\;\;  \mbox{\rm (from \eqref{eqn_r21})}
\\
\label{eqn_r3}
& \leq -\left({1\over m}+{\beta_k\gamma^2_{k-1}\over 2\gamma_k}\right)\|y_k-x^*\|^2 + {\beta_k\gamma_{k-1}^2\over 2\gamma_k} \|x_k-x^*\|^2.
\end{align}
By substituting \eqref{eqn_r1}, \eqref{eqn_r2}, and \eqref{eqn_r3}
into \eqref{eqn_r}, we obtain
\begin{align}
\nonumber
& \mathbb{E}_{i(k)|I(k-1)}(r^2_{k+1})\\
\nonumber
& \leq \beta_k\|v_k-x^*\|^2_{(A^TA)^+} + {\gamma_k\over m}\|y_k-x^*\|^2 \\
\nonumber
& \quad + {\gamma_k^2}(\|y_k-x^*\|^2-\mathbb{E}_{i(k)|I(k-1)}(\|x_{k+1}-x^*\|^2))\\
\nonumber
& \quad -\left({2\gamma_k\over m}+{\beta_k\gamma^2_{k-1} }\right)\|y_k-x^*\|^2 + {\beta_k\gamma_{k-1}^2} \|x_k-x^*\|^2\\
\nonumber
& \leq \beta_k\|v_k-x^*\|^2_{(A^TA)^+} - {\gamma_k^2}\mathbb{E}_{i(k)|I(k-1)}(\|x_{k+1}-x^*\|^2) + {\beta_k\gamma_{k-1}^2}\|x_k-x^*\|^2\\
\nonumber
& \quad +\left({\gamma_k^2}-{\gamma_k\over m} - {\beta_k\gamma_{k-1}^2} \right)\|y_k-x^*\|^2 \\
\label{eqn_thmproof1}
& =  \beta_k\|v_k-x^*\|^2_{(A^TA)^+} -
{\gamma_k^2}\mathbb{E}_{i(k)|I(k-1)}(\|x_{k+1}-x^*\|^2) +
{\beta_k\gamma_{k-1}^2}\|x_k-x^*\|^2,
\end{align}
where the final equality is a consequence of \eqref{eqn_r20}.

We now define two scalar sequences $\{A_k\}$ and $\{B_k\}$ as follows:
\begin{equation}
  A_k\geq 0,~B_k\geq 0,~B_0\neq 0,~B_{k+1}^2 = {B_k^2\over\beta_k},
  ~A_{k+1}^2=\gamma_{k}^2 B_{k+1}^2. \label{eqn_thmproof1_5}
\end{equation}
We set $A_0 = 0$ (to be consistent with the definition
\eqref{eqn_thmproof1_5} and the fact that $\gamma_{-1}=0$ in
Algorithm~\ref{alg_arka}) and note that $B_{k+1}\geq B_k$, since
$\beta_k \in (0,1]$. Since from \eqref{eqn_thmproof1_5} together with
\eqref{eqn_alg1} and \eqref{eqn_alg3}, we have
\[
A^2_{k+1}=\frac{B_k^2\gamma_k^2}{\beta_k}=
\frac{\gamma_k^2 A_k^2}{\beta_k \gamma_{k-1}^2}=
\frac{A_k^2\gamma_k^2}{\gamma_k^2-\gamma_k/m},
\]
we obtain that $\{A_k\}$ is also an increasing sequence.

Multiplying the last inequality \eqref{eqn_thmproof1} by $B_{k+1}^2$,
and using the definition of $r_k$ \eqref{eqn_defr} along with
\eqref{eqn_thmproof1_5} (in particular, the identities $B_{k+1}^2
\gamma_k^2 = A_{k+1}^2$, $B_{k+1}^2 \beta_k = B_k^2$, and $B_{k+1}^2
\beta_k \gamma_{k-1}^2 = A_k^2$), we obtain
\begin{align}
\nonumber
B_{k+1}^2\mathbb{E}_{i(k)|I(k-1)}(r_{k+1}^2) +
A_{k+1}^2\mathbb{E}_{i(k)|I(k-1)} & (\|x_{k+1}-x^*\|^2) \\
& \leq B_{k}^2 r_k^2 + A^2_{k}\|x_k-x^*\|^2. \label{eqn_proof_key}
\end{align}
It follows that
\begin{align*}
\mathbb{E}_{I(k)}(B_{k+1}^2 & r_{k+1}^2 +
A_{k+1}^2(\|x_{k+1}-x^*\|^2))\\
& = \mathbb{E}_{I(k-1)}(B_{k+1}^2\mathbb{E}_{i(k)|I(k-1)}(r_{k+1}^2) +
A_{k+1}^2\mathbb{E}_{i(k)|I(k-1)}(\|x_{k+1}-x^*\|^2))\\
& \leq \mathbb{E}_{I(k-1)}(B_{k}^2r_{k}^2 + A_{k}^2\|x_{k}-x^*\|^2).
\end{align*}
By applying this inequality recursively, we obtain
\begin{align*}
\mathbb{E}_{I(k)}(B_{k+1}^2r_{k+1}^2 +
A_{k+1}^2(\|x_{k+1}-x^*\|^2))
& \leq\mathbb{E}_{I(0)}(B_{1}^2r_{1}^2 +
A_{1}^2\|x_{1}-x^*\|^2)\\
& \leq B_0^2r_0^2 + A_0^2\|x_0-x^*\|^2=
B_0^2r_0^2,
\end{align*}
where we dropped the last term because $A_0=0$.  It follows from this
bound that
\begin{equation}
\mathbb{E}(r_{k+1}^2)\leq {B_0^2\over
B_{k+1}^2}r_0^2 \quad
\mbox{\rm and} \quad
\mathbb{E}(\|x_{k+1}-x^*\|^2)\leq
{B^2_0\over A_{k+1}^2}r^2_0. \label{eqn_thmproof4}
\end{equation}

We now need to estimate the growth of two sequences $\{A_k\}$ and
$\{B_k\}$. Here we follow the proof for the accelerated coordinate
descent algorithm of \cite{Nesterov10}, but spelling out some details
skipped in that paper. We have
\[
B_{k}^2=B_{k+1}^2\beta_k= \left( 1-{\mumin\over
  m}\gamma_k \right) B_{k+1}^2=\left(1-{\mumin A_{k+1}\over
  m B_{k+1}} \right) B_{k+1}^2,
\]
 which implies that
\[
{\mumin\over m}A_{k+1}B_{k+1}= B^2_{k+1}-B^2_k=
(B_k+B_{k+1})(B_{k+1}-B_k),
\]
so by recalling that $B_{k+1}\geq B_k$, we obtain
\begin{equation}
  B_{k+1} \geq B_k + {\mumin\over 2m}A_k.
  \label{eqn_thmproofb}
\end{equation}
We have
\begin{alignat*}{2}
{A^2_{k+1}\over B^2_{k+1}}-{A_{k+1}\over
B_{k+1}m}
& =\gamma_k^2-{\gamma_k\over m} \quad && (\text{from ~\eqref{eqn_thmproof1_5}})\\
& =(1-{\gamma_k\mumin \over
m})\gamma_{k-1}^2 \quad && (\text{from ~\eqref{eqn_alg1}})\\
&={\beta_k A_k^2\over B_k^2}={A^2_{k}\over B^2_{k+1}} \quad && (\text{from
~\eqref{eqn_alg3} and \eqref{eqn_thmproof1_5}}),
\end{alignat*}
so we obtain by multiplying both sides of this expression by
$B_{k+1}^2$ and using $A_{k+1}\geq A_k$ that
\[
  {1\over m} A_{k+1}B_{k+1}= A_{k+1}^2-A_k^2
=(A_{k+1}+A_k)(A_{k+1}-A_k)\leq
  2A_{k+1}(A_{k+1}-A_k)
\]
and therefore
\begin{equation}
  \label{eqn_thmproofa}
A_{k+1}\geq A_k+{B_{k+1}\over 2m} \ge A_k + \frac{B_k}{2m}.
\end{equation}
By combining the inequalities~\eqref{eqn_thmproofb} and
\eqref{eqn_thmproofa} and applying a recursive argument, we can
estimate $A_{k+1}$ and $B_{k+1}$ as follows:
\[
  \left[
    \begin{array}{cc}
      A_{k+1} \\
      B_{k+1}
    \end{array}
  \right] \geq
  \left[
        \begin{array}{cc}
        1 & {1\over 2m} \\
        {\mumin\over 2m} & 1 \\
        \end{array}
    \right]^{k+1}
    \left[
      \begin{array}{c}
        A_0 \\
        B_0 \\
      \end{array}
    \right].
\]
The Jordan decomposition of the matrix in this expression is
\[
\left[
        \begin{array}{cc}
        1 & {1\over 2m} \\
        {\mumin\over 2m} & 1 \\
        \end{array}
    \right] =
\left[ \begin{matrix} 1 & 1 \\ \sqrt{\mumin} & -\sqrt{\mumin} \end{matrix} \right]^{-1}
\left[ \begin{matrix} \sigma_1 & 0 \\ 0 & \sigma_2 \end{matrix} \right]
\left[ \begin{matrix} 1 & 1 \\ \sqrt{\mumin} & -\sqrt{\mumin} \end{matrix} \right],
\]
with
\[
\sigma_1=1+ {\sqrt{\mumin}\over 2m}, \quad
\sigma_2=1-{\sqrt{\mumin}\over 2m}.
\]
Thus we have
\begin{align*}
\left[ \begin{matrix} 1 & \frac{1}{2m} \\ \frac{\mumin}{2m} & 1 \end{matrix}
\right]^{k+1}
&=
\frac12
\left[ \begin{matrix} 1 & \frac{1}{\sqrt{\mumin}} \\
1 & -\frac{1}{\sqrt{\mumin}} \end{matrix} \right]
\left[ \begin{matrix} \sigma_1^{k+1} & 0 \\
0 & \sigma_2^{k+1} \end{matrix} \right]
    \left[
        \begin{array}{cc}
        1 & 1 \\
        \sqrt{\mumin} & -\sqrt{\mumin} \\
        \end{array}
    \right] \\
 &=\frac12 \left[
  \begin{array}{cc}
    \sigma_1^{k+1}+\sigma_2^{k+1} & (\sigma_1^{k+1}-\sigma_2^{k+1})/\sqrt{\mumin}  \\
    (\sigma_1^{k+1}-\sigma_2^{k+1})\sqrt{\mumin} & \sigma_1^{k+1}+\sigma_2^{k+1} \\
  \end{array}
\right]
\end{align*}
which implies $A_{k+1}\geq
B_0 (\sigma_1^{k+1}-\sigma_2^{k+1}) / (2\sqrt{\mumin})$ and
$B_{k+1}\geq (\sigma_1^{k+1}+\sigma_2^{k+1})B_0/2$.
By combining these bounds with
\eqref{eqn_thmproof4}, we obtain
\begin{align*}
\mathbb{E} (r_{k+1}^2) &
= \mathbb{E} (\|v_{k+1}-x^*\|_{(A^TA)^+}^2)
\le \frac{B_0^2}{B_{k+1}^2} r_0^2
\le \frac{4 \|x_0-x^*\|_{(A^TA)^+}^2}{(\sigma_1^{k+1} + \sigma_2^{k+1})^2},\\
\mathbb{E} (\|x_{k+1}-x^*\|^2) & \le \frac{B_0^2}{A_{k+1}^2} r_0^2
\le \frac{4 \mumin \|x_0-x^*\|_{(A^TA)^+}^2}{(\sigma_1^{k+1} - \sigma_2^{k+1})^2},
\end{align*}
completing the proof.
\end{proof}

\end{document}